%% file: main.tex
\documentclass[11pt]{article}
\usepackage[utf8]{inputenc}

\usepackage{authblk}

\title{Vehicle Routing Problems in the Age of Semi-Autonomous Driving}
\author[1]{Hins Hu}
\author[2]{Samitha Samaranayake}
\affil[1]{{\small Systems Engineering, Cornell University, zh223@cornell.edu}}
\affil[2]{{\small School of Civil and Environmental Engineering, Cornell University, samitha@cornell.edu}}
\date{}

\usepackage[dvipsnames]{xcolor}
\definecolor{cornellred}{HTML}{B31B1B}

\usepackage{tikz}
\usepackage{tikz-network}
\usetikzlibrary{backgrounds}
\usetikzlibrary{patterns}

\usepackage{pgfplots}
\pgfplotsset{width=10cm,compat=1.9}

\usepackage{amsmath}
\usepackage{amssymb}
\usepackage{amsthm}
\usepackage{bm}
\usepackage{mathtools}

\usepackage{thmtools, thm-restate}
\usepackage{optidef}

\usepackage{graphicx}

\usepackage{amsfonts}

\usepackage{algorithm}
\usepackage{algpseudocode}

\usepackage[style = numeric]{biblatex}
\newtheorem{example}{Example}[section]

\newtheorem{theorem}{Theorem}[section]

\newtheorem{proposition}[theorem]{Proposition}

\usepackage{caption}
\usepackage{subcaption}
\usepackage{longtable}
\usepackage[top = 1in, bottom = 1in, left = 1in, right = 1in]{geometry}
\usepackage{multirow}
\usepackage{setspace}
\usepackage[colorlinks=true,linkcolor=black,citecolor=black,urlcolor=black,
            pdfauthor={},pdftitle={},pdfproducer={},pdfcreator={}]{hyperref}

\begin{document}

\maketitle

\begin{abstract}
We are in the midst of a semi-autonomous era in urban transportation in which varying forms of vehicle autonomy are gradually being introduced. This phase of partial autonomy is anticipated by some to span a few decades due to various challenges, including budgetary constraints to upgrade the infrastructure and technological obstacles in the deployment of fully autonomous vehicles (AV) at scale. In this study, we introduce the \textit{vehicle routing problem in a semi-autonomous environment} (VRP-SA) where the road network is not fully AV-enabled in the sense that a portion of it is either not suitable for AVs or requires additional resources in real-time (e.g., remote control) for AVs to pass through. Moreover, such resources are scarce and usually subject to a budget constraint. An exact mixed-integer linear program (MILP) is formulated to minimize the total routing cost of service in this environment. We propose a two-phase algorithm based on a family of \textit{feasibility recovering sub-problems} (FRP) to solve the VRP-SA efficiently. Our algorithm is implemented and tested on a new set of instances that are tailored for the VRP-SA by adding stratified grid road networks to the benchmark instances. The result demonstrates a reduction of up to 37.5\% in vehicle routing costs if the fleet actively exploits the AV-enabled roads in the environment. Additional analysis reveals that cost reduction is higher with more budget and longer operational hours.
\end{abstract}

\section{Introduction}
With the development of numerous advanced technologies related to autonomous driving, such as smart sensors, computer vision, and high-speed wireless communication, the deployment of autonomous vehicles (AV) fleets in all kinds of transportation is becoming more and more viable. More importantly, replacing human-driven vehicles (HDV) with AVs in many businesses may potentially reduce the operational cost dramatically as drivers' wages contribute a great portion of it. In the trucking industry, according to a report \cite{leslie2022analysis}  by the American Transportation Research Institute, the driver-based cost contributed 35.8\% of the average marginal cost per mile in 2011 and this number climbed up to 40.2\% in 2022 due to increases in labor costs. If we deduct the amortized cost of vehicle purchases and leases, the percentage of the driver-based cost to the operational cost per mile in 2011 and 2019 are 37.6\% and 47.2\% respectively.

Although the AV industry has seen significant advancements and substantial investment, a fully autonomous driving environment remains years away. The era of \textit{semi-autonomous driving} is expected to persist for several decades due to numerous barriers. First, key technologies such as motion planning and object detection continue to pose significant bottlenecks. According to SAE International’s classification system \cite{sae2018taxonomy}, autonomous driving is categorized into six levels, ranging from level 0 (no automation) to level 5 (full automation). As of March 2021, Honda became the first manufacturer to offer a legally approved level-3 vehicle \cite{Honda2021}, capable of autonomous driving in certain conditions, though human intervention is still required in some situations. Second, mixed traffic environments, with both AVs and HDVs, are inevitable due to factors such as consumer preference and the long service life of traditional automobiles. A survey on commuting preferences \cite{HABOUCHA201737} conducted with 721 participants found that many individuals remain hesitant to adopt AVs. Even with a fully subsidized shared AV service, 44\% of respondents still preferred regular vehicles, and only 75\% expressed interest in switching to shared AVs. Third, achieving full automation in some scenarios is more challenging than in others. For example, high-density urban areas, compared to long-haul truck platooning, present greater complexity for level-5 automation due to intricate road networks, frequent interactions between vehicles and pedestrians, and the existence of traffic control. Fourth, autonomous driving requires significant infrastructure upgrades, particularly in vehicle-to-infrastructure (V2I) communication systems, to facilitate precise AV maneuvers such as lane-changing and merging. However, the investment required for these upgrades is often constrained by budgetary limitations, making it a long-term undertaking. Finally, legislative and policy frameworks tend to lag behind technological development. Concerns around safety, equity, and privacy often drive a more conservative approach to regulating autonomous driving, further delaying widespread adoption.

Given the limitations aforementioned, new challenges arise when deploying AVs in a semi-autonomous environment. Specifically, road networks are not fully accessible to AV fleets in the sense that certain levels of AVs may be prohibited from traversing a subset of road segments, leading to a classification of roads based on their compatibility with different levels of AV autonomy. In the simplest case, the road network can be modeled as a bi-level system comprising \textit{ordinary roads} and \textit{AV-enabled roads}. Additionally, a critical feature of the semi-autonomous environment is that not all AVs operate at full autonomy (i.e., level 5). Some may require external resources, such as real-time remote control, to navigate AV-enabled roads. These resources, however, are typically costly. For instance, in October 2020, Waymo launched a geo-fenced semi-autonomous ride-hailing service in Phoenix, U.S., supported by a team of remote engineers available to intervene in contingencies \cite{Waymophoenix}. Similarly, Enride has deployed electric autonomous trucks on low-traffic private roads within logistics and manufacturing hubs and is working toward operating them on public roads with remote supervision \cite{enride2022}.

With that said, optimal vehicle dispatching and routing strategies at the operational level, traditionally modeled by the vehicle routing problem (VRP), need to be re-examination and re-designed, particularly for high-density urban areas. Moreover, vehicle dispatching schedules are crucial because the simultaneous demand for resources by all AVs must not exceed the maximum capacity. Exceeding this capacity can result in system failures with severe consequences, such as collisions. From a resource allocation perspective, fleet operators need to assign resources to AVs effectively by designing vehicle schedules based on prior knowledge of road networks. To address these new challenges, a holistic optimization approach is essential. Therefore, we propose a novel variant of the VRP, termed the \textit{Vehicle Routing Problem in the Semi-Autonomous Environment} (VRP-SA).

\textbf{The contribution of our work is threefold:} (1) We formalize a novel vehicle routing problem called VRP-SA to address the new challenges of deploying AV fleets in a semi-autonomous environment. (2) We provide a practically efficient solution approach for the VRP-SA. (3) We analyze the impact of varying unit routing costs and the density of AV-enabled roads on the deployment of AVs in the age of semi-autonomous driving.

The paper is organized as follows: Section \ref{sec:literature} briefly reviews the history of VRP and relevant literature on VRP variants involving mixed fleets and limited resources. Section \ref{sec:description} presents a formal description of the VRP-SA. Section \ref{sec:milp} introduces an exact mixed-integer linear program (MILP) for the VRP-SA. Section \ref{sec:milp_extend} discusses some extended studies in the MILP to enhance the model's comprehensiveness and adaptability. Section \ref{sec:algo} presents a two-phase algorithm based on a family of feasibility recovering sub-problems (FRP) to solve the VRP-SA efficiently. Section \ref{sec:exp} presents a new set of instances tailored for the VRP-SA and analyzes the results of numerical experiments conducted on these instances. Section \ref{sec:conclusion} concludes the paper and outlines potential directions for future research.

\section{Literature Review} \label{sec:literature}
VRP was first formulated as an integer linear program by Dantzig and Ramser \cite{dantzig1959truck} in 1959. The seminal heuristic algorithm to solve it was developed by Clark and Wright \cite{clarke1964scheduling} in 1964. The proof of the NP-hardness of VRP was provided by Lenstra and Kan \cite{lenstra1981complexity} in 1981. In addition to these milestones, extensive studies in VRP models (e.g. \cite{golden1984fleet}
\cite{kolen1987vehicle} \cite{savelsbergh1995general}
\cite{renaud1996tabu} \cite{erdougan2012green}
\cite{schneider2014electric}), construction heuristics (e.g. \cite{gillett1974heuristic} \cite{fisher1981generalized} \cite{bramel1995location}), meta-heuristics (e.g. \cite{homberger1999two} \cite{bell2004ant} \cite{pisinger2007general}), branch-and-bounds approaches (i.e., \cite{desrochers1992new} \cite{fischetti1994branch} \cite{fukasawa2006robust}), and learning-based approaches (e.g. \cite{nazari2018reinforcement} \cite{kool2018attention} \cite{li2021learning}) have been conducted and published over the past decades. Numerous benchmark data sets (e.g. \cite{solomon1987algorithms} \cite{augerat1995computational} \cite{uchoa2017new}) have also been created to support the performance tests of newly-developed VRP algorithms. The most recent research on VRP follows a trend toward designing a general-purpose algorithm by the hybridization of various frameworks to solve a wide range of VRP variants efficiently and on a larger scale. The representative work is the unified hybrid genetic search (UHGS) developed by Vidal et al. \cite{VIDAL2014658}. More comprehensive literature reviews from different aspects of VRP are referred to these papers (\cite{braysy2005vehicle_1} \cite{braysy2005vehicle_2} \cite{montoya2015literature} \cite{kocc2016thirty} \cite{lin2014survey} \cite{braekers2016vehicle}) and this book \cite{toth2014vehicle}.

A key feature of the VRP-SA is its mixed fleet of AVs and HDVs, which can be modeled by the family of Heterogeneous Fleet Vehicle Routing Problems (HFVRP) involving only two types of vehicles. This topic is covered in a chapter by Toth and Vigo \cite{toth2014vehicle}. The HFVRP family can be further refined into different categories depending on fleet constraints (limited or unlimited) and cost structures (vehicle-dependent or vehicle-independent). Research on exact approaches for the HFVRP is relatively limited. The most effective exact approach currently available is the branch-and-cut-and-price algorithm developed by Baldacci and Mingozzi \cite{baldacci2009unified}. In contrast, most existing approaches are heuristic-based, such as the multi-start adaptive memory programming (MAMP) by Li et al. \cite{LI20101111}, the variable neighborhood search (VNS) by Imran et al. \cite{IMRAN2009509}, and the memetic algorithms by Prins \cite{PRINS2009916}. Currently, the UHGS approach \cite{VIDAL2014658} is recognized as the leading method for the HFVRP.

One recent work by Molina et al. \cite{MOLINA2020103745} also discussed the VRP with a limited number of resources available. However, their focus is on scenarios where resource limitations—such as drivers, vehicles, or other equipment—prevent the servicing of all customers. This contrasts with our work, where limited resources (e.g., remote control) shared across the entire AV fleet can influence the routing decisions of individual vehicles when serving multiple customers.

\section{The Problem Statement of VRP-SA} \label{sec:description}
The VRP-SA admits the following inputs: (1) A directed graph $G = (V, E = E^a \cup E^o)$ representing the underlying road network, where $V$ is the set of nodes representing road intersections, and $E^a$ and $E^o$ are the sets of edges representing AV-enabled and ordinary road segments, respectively. (2) A depot $o \in V$ from which all vehicles depart and to which they return after the service. (3) A time horizon $[0, T]$ of operations (e.g., a day). (4) A set of customers $D \subset V$ with a demand vector $\bm{d} \in \mathbb{R}^{|D|}_{++}$. (5) A mixed vehicle fleet $M = M^a \cup M^h$, where $M^a$ is the set of AVs and $M^h$ is the set of HDVs, all with identical capacity $W \in \mathbb{R}_{++}$. (6) A pair of fixed costs $(f^a, f^h)$ of dispatching an AV and an HDV, respectively. (7) A vector of routing costs $\bm{c} \in \mathbb{R}_{++}^{3 \times |E|}$, where each cost is associated with a road segment in $E$ and varies by vehicle type and road type. The cost structure is detailed in Table \ref{tab:cost}. Here, $c_e^1$ and $c_e^2$ represent the routing costs for an AV traveling on an AV-enabled road segment $e \in E^a$ and an ordinary road segment $e \in E^o$, respectively. Similarly, $e_e^0$ denotes the routing cost for an HDV traversing any road $e \in E$, regardless of type. Two constant cost adjustment factors, $\eta_1$ and $\eta_2$, are applied uniformly across all edges. (8) A vector of travel times $\Delta \bm{t} \in \mathbb{R}_{++}^{|E|}$, with each entry corresponding to a road segment. (9) A budget $B \in \mathbb{N}_+$ representing the number of remote controllers available to assist AVs in real-time while they are traversing ordinary road segments. All aforementioned notations are summarized in Table \ref{tab:notations} for reference in later sections.
\begin{table}[!htbp]
    \centering
    \begin{tabular}{|c|c|c|}
        \cline{1-3}
         & AV-Enabled Roads & Ordinary Roads\\ \hline
        AVs & $c^{1}_e = \eta_1 \cdot c^0_e$ & $c^{2}_e = \eta_2 \cdot c^0_e$ \\ \hline
        HDVs & $ c^0_e$ & $ c^0_e$ \\ \hline
    \end{tabular}
    \caption{The routing cost structure of a vehicle traversing road segment $e$}
    \label{tab:cost}
\end{table}

The problem is subject to the following assumptions: (1) HDVs can traverse the entire road network freely without external resources, whereas AVs require additional support from remote controllers to pass through ordinary roads. Specifically, an AV locks one remote controller upon entering an ordinary road from an AV-enabled road and releases it upon returning to an AV-enabled road. (2) As outlined in Table \ref{tab:cost}, the routing cost for a vehicle on a given road segment depends on both the vehicle and road types. We further assume that, in the age of semi-autonomous driving, the routing cost for an AV on an AV-enabled road is lower than that for an HDV on the same road (i.e., $\eta_1 < 1$). Conversely, on an ordinary road, the cost for an AV exceeds that of an HDV (i.e., $\eta_2 > 1$). This assumption is justified by the significant cost reduction from eliminating driver-related expenses on AV-enabled roads. However, on ordinary roads, the scarcity and high cost of remote controllers lead to increased operational expenses, surpassing even the driver-related portion. (3) The dispatched times for vehicles are flexible within $[0, T]$, provided that all vehicles return to the depot before $T$.

The objective of the VRP-SA is to find a minimum-cost dispatching and routing strategy that serves all customers while satisfying the budget constraint of remote controllers, the return time constraints, and the capacity constraints. The fixed costs for establishing the real-time remote control system are treated as overheads and thus excluded from our decision making. Additionally, we do not integrate other commonly seen conventional constraints, such as time windows for customers and the order of pick-up and delivery, as they are relatively independent of the core issues addressed in this paper.

\section{The Mixed-Integer Linear Program} \label{sec:milp}
In this section, we will first briefly discuss a challenge of formulating the VRP-SA as a mixed-integer linear program (MILP). Next, we will establish a connection between a key characteristic of the VRP-SA and a simpler classical combinatorial optimization problem known as the \textit{Steiner Traveling Salesman Problem} (STSP). Leveraging several properties of the STSP, we will then demonstrate how to construct an \textit{expanded graph} based on the original road network and derive an exact MILP formulation for our problem on this expanded graph.

Many VRP models and algorithms are formulated and designed based on the \textit{metric closure} of the underlying road network \cite{toth2014vehicle}. The metric closure of a graph $G$ on a subset of nodes $D$ is a complete graph $\overline{G}$ on $D$, where each edge is weighted by the cost of the shortest path between the two corresponding nodes in $G$. This transformation allows each vehicle route in the solution to be represented as an ordered sequence, indicating the order of serving the subset of customers assigned to that vehicle. The low-level routing in the real-world road network is pre-determined by computing the minimum-cost paths between all pairs of customers. This trick significantly reduces the number of decision variables and constraints required to model a VRP, resulting in a more compact formulation. For instance, it ensures that each customer node is visited exactly once in the optimal solution, enabling the use of a constraint akin to unit flow conservation \cite{dantzig1959truck}. This pre-processing step, which transforms the underlying network into its metric closure, is illustrated in Figure \ref{fig:metric_closure}.

However, this trick is not applicable to the VRP-SA due to the presence of mixed road types in the underlying network. While shortest paths can still be pre-determined, a vehicle may not always travel along the shortest path between two customers in the optimal solution. This is because its ability to navigate an ordinary road segment depends on the continuous availability of a remote controller during the entire period it remains on the segment, from entry to exit. Consequently, the feasibility of selecting ordinary road segments for different AV routes is interdependent and constrained by a real-time enforced budget. Figure \ref{fig:metric_closure} further demonstrates that, once the network is transformed into its metric closure, all information about the road type of each edge is lost.
\begin{figure}[htbp!]
    \centering
    \includegraphics[width=0.9\linewidth]{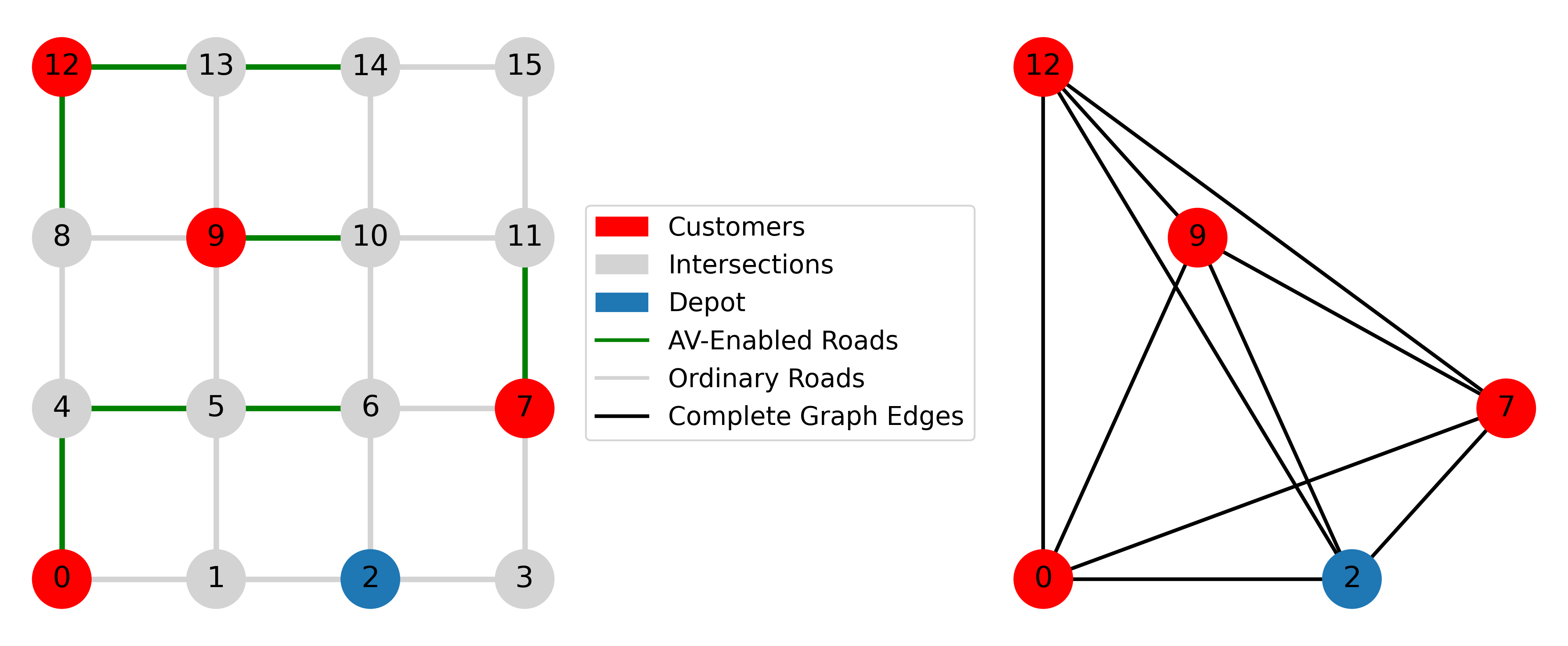}
    \caption{A toy example of a road network of mixed road types and its metric closure}
    \label{fig:metric_closure}
\end{figure}

Therefore, to formulate an MILP for the VRP-SA, we must work directly on the original road network. A key characteristic of an optimal solution to the VRP-SA in the original network is that any node may be visited more than once if necessary, which features the STSP in its most basic form. The formal definition of the STSP is as follows: Given an undirected and connected graph, a subset of nodes designated as required nodes to be served, a distinguished required node serving as the depot, and positive costs associated with all edges, the objective is to determine a minimum-cost route that visits all required nodes. Essentially, the Steiner version of the TSP relaxes the requirement for the graph to be complete \cite{RODRIGUEZPEREIRA2019615}. We identify the following two lemmas of the STSP, with formal proofs provided in Appendix \ref{app: A}.

\begin{restatable}{lemma}{lemmanode} \label{lemma:1}
In an optimal solution to the STSP in a directed graph, a node can be visited at most $n$ times, where $n$ is the number of required nodes in the graph. 
\end{restatable}

\begin{restatable}{lemma}{lemmaedge} \label{lemma:2}
In an optimal solution to the STSP in a directed graph, an edge can be traversed at most $n-1$ times, where $n$ is the number of required nodes in the graph.
\end{restatable}

The observations in Lemma \ref{lemma:1} and Lemma \ref{lemma:2} can be naturally extended to the VRP-SA. Specifically, for each node in the set \( V \), an upper bound on the number of times it can appear in an optimal solution to the VRP-SA can be established. This result is presented in Proposition \ref{prop: k}, with a detailed proof provided in Appendix \ref{app: A}.

\begin{restatable}{proposition}{propositionk} \label{prop: k}
Let $P(m)$ be the route of vehicle $m$ in the optimal solution to the VRP-SA. Let $W$ be the capacity of vehicle $m$. Let $\{d_i: \; i = 1, \dots, |D| \}$ be a non-decreasing sequence indicating the demands of customer set $D$. Then, a node $v \in P(m)$ can be visited by vehicle $m$ at most $k + 1$ times, and an edge $(u, v), \; \forall \; u, v \in P(m)$ can be traversed by vehicle $m$ at most $k$ times, where $k$ is defined as the largest integer $j = 1, \dots, |D|$ satisfying $\sum_{i=1}^j d_i \leq W$.
\end{restatable}

Building on the result of Proposition \ref{prop: k}, we propose an approach to formulate an exact MILP for the VRP-SA by constructing a \textit{$k$-layer expanded graph} $G_e$ from the original network graph $G$, where $k$ is the number defined in Proposition \ref{prop: k}. The construction process is as follows:
\begin{enumerate}
    \item Duplicate $G$ by $k$ times and stack them in layers.
    \item Create an dummy node $s$ called the \textit{sink depot} on the base layer. 
    \item For every two adjacent layers, add an dummy directed edge pointing from every customer on the lower layer to its duplicate customer on the upper layer.
    \item For every duplicate depot on the duplicate layers, add an dummy directed edge pointing from the depot to the sink depot $s$.
    \item Assign zero cost and zero travel time to each dummy edge. 
\end{enumerate}

Accordingly, additional notations are introduced to account for duplication in the expanded graph and to support the MILP formulation later in this section. These notations are clearly explained and summarized in Table \ref{tab:notations}, so their definitions will only be repeated in the text when necessary. 
\begin{table}[!t]
    \centering
    \renewcommand{\arraystretch}{1.05}
    \begin{tabular}{l|l} \hline
    \textbf{Notation} & \textbf{Definition} \\ \hline
    $G$ & The graph representing the underlying road network  \\
    $V$ and $E$ & The entire set of nodes and edges in graph $G$ \\
    $E^a$ and $E^o$ & The set of edges representing AV-enabled and ordinary roads \\
    $D$ & The set of customers \\
    $\bm{d}$ & The demand vector of customers \\
    $M$ & The entire set of vehicles \\
    $M^a$ and $M^h$ & The set of AVs and HDVs \\
    $W$ & The uniform capacity of all vehicles \\
    $\bm{c}$ & The vector of edge-wise routing costs \\
    $f^a$ and $f^h$ & The fixed costs of dispatching an AV and an HDV \\ 
    $\eta_1$ and $\eta_2$ & The cost adjustment factors for AVs on two types of roads \\
    $\Delta\bm{t}$ & The vector of edge-wise travel times \\
    $T$ & The end time of operation \\
    $G_e$ & The expanded graph \\
    $k$ & The number of layers in the expanded graph \\
    $o$ and $s$ & The source and sink depot on the base layer of $G_e$ \\
    $V_e$ and $E_e$ & The set of nodes and edges in $G_e$ \\
    $E_e^a$ and $E_e^o$ & The set of AV-enabled and ordinary edges in $E_e$ \\
    $D_e$ & The set of customers in $G_e$, including duplicates \\
    $O$ & The set of depots in $G_e$, including duplicates \\
    $D^H(i)$ & The set of duplicate customers on layer $i \in [k]$ \\
    $D^V(i)$ & The set of customers duplicated from a single customer $i \in D$ \\
    $A$ & The set of dummy edges connecting layers in $G_e$ \\
    $\delta^-(i)$ and $\delta^+(i)$ & The set of outgoing and incoming neighbor nodes of $i \in V_e$ \\
    $Q$ & The set of discretized time intervals in $[0, T]$ \\
    $a_q$ and $b_q$ & The start and end time of interval $q \in Q$ \\
    $B$ & The budget of remote controllers \\
    $S$ & The Cartesian product of sets $M^a$ and $E_e^o$ \\ \hline
    \end{tabular}
    \caption{Notations}
    \label{tab:notations}
\end{table}

Figure \ref{fig:expanded_graph} visualizes an example of a $3$-layer expanded graph of a small undirected road network. Two upper layers in dashed lines are duplicated from the base layer in solid lines. The dots in sky blue are depots in $O$. The dots in light purple are customers in $D_e$. The black dots are road intersections in $V_e \setminus (O \cup D_e)$. The bold dashed arrows are dummy edges in $A$.

\begin{figure}[!htbp]
    \centering
    \captionsetup{justification=centering}
    \begin{subfigure}[t]{0.49\textwidth}
        \centering
        \includegraphics[width=0.8\linewidth]{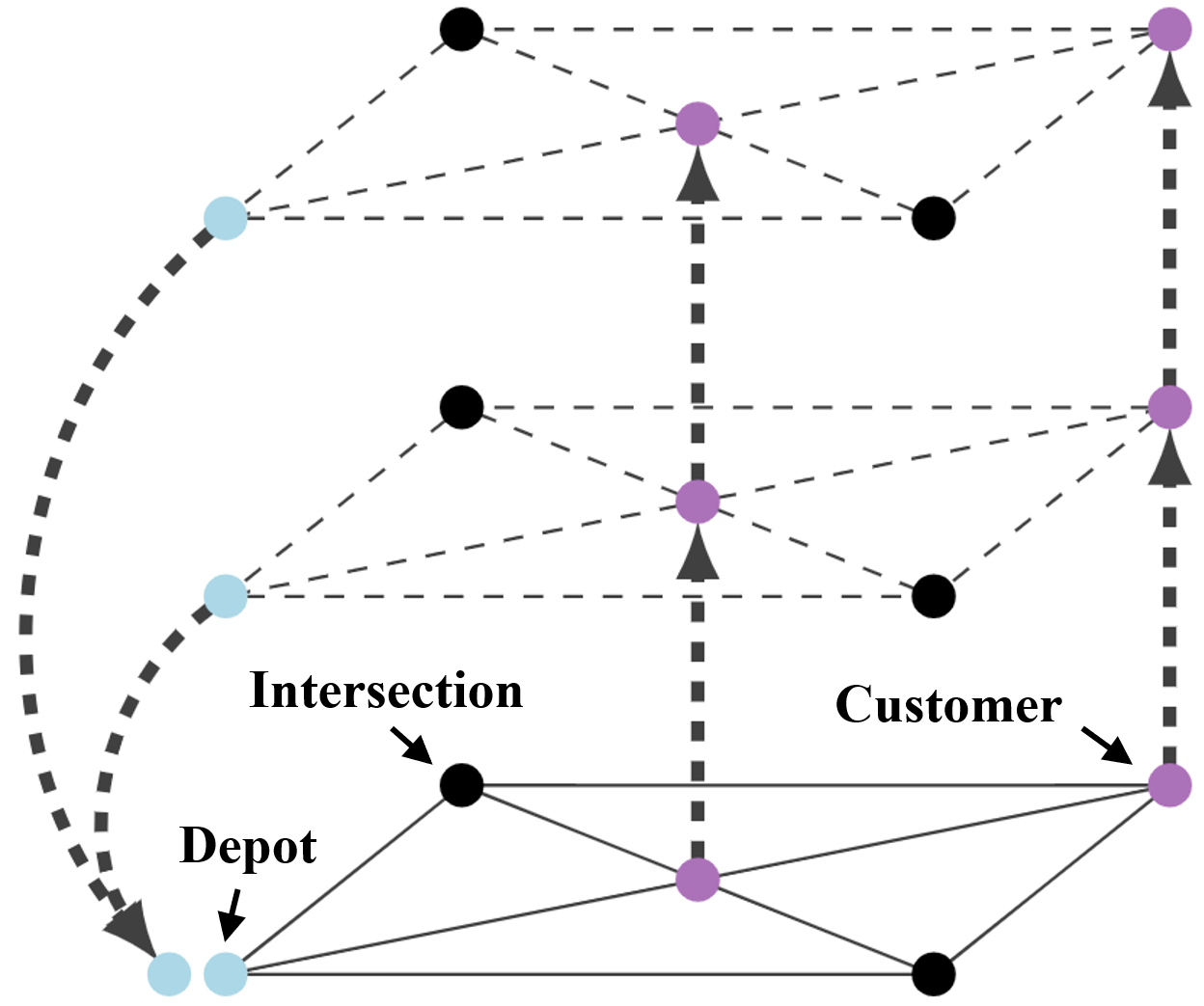}
        \caption{The 3-layer expanded graph of a tiny undirected road network with $2$ customers}
        \label{fig:expanded_graph}
    \end{subfigure}
    \begin{subfigure}[t]{0.49\textwidth}
        \centering
        \includegraphics[width=0.8\linewidth]{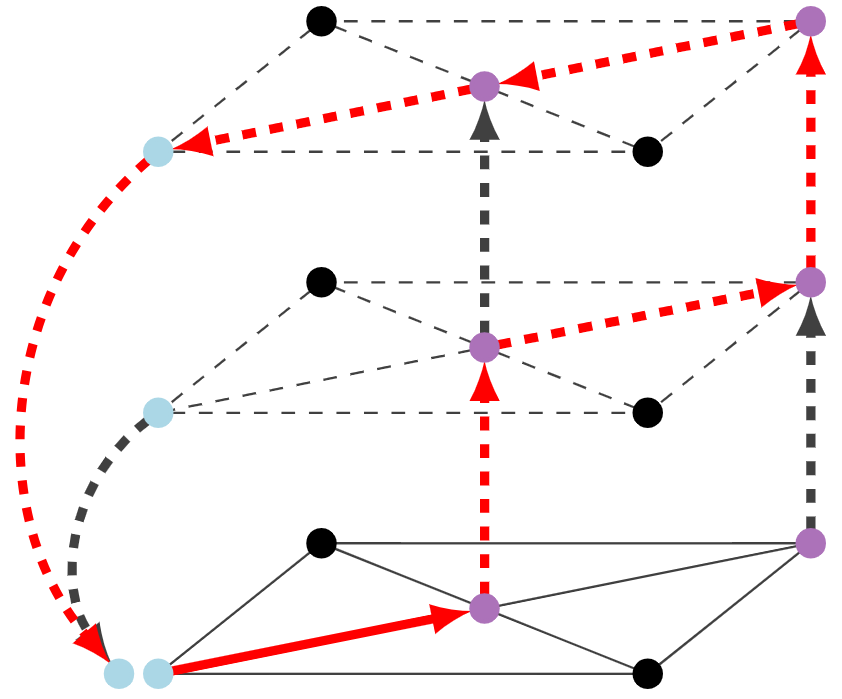}
        \caption{An optimal route highlighted in red for the expanded graph in Figure \ref{fig:expanded_graph}}
        \label{fig:optimal_route}
    \end{subfigure}
    \caption{Visualization of a $3$-layer expanded graph}
    \label{fig:expanded_graph_overview}
\end{figure}

The necessity of constructing the $k$-layer expanded graph arises from two key considerations: (1) each node or edge in the network, except the depot, can be visited up to $k$ times by a vehicle, and (2) we have to track the timestamps whenever a vehicle transitions between ordinary roads and AV-enabled roads so that we can count the total number of occupied remote controllers at any given moment. The expanded graph provides an additional dimension to distinguish between different times a vehicle visits the same node or traverses the same edge. It is important to note that this expanded graph is not \textit{strongly connected}, as nodes on lower layers are not reachable from nodes on upper layers. Thus, if we can introduce a properly designed constraint in the MILP that forces a vehicle to move up one layer via an dummy edge every time it serves a customer, each node in $V_e$ is guaranteed to be visited by the vehicle at most once in the optimal solution to the VRP-SA. Figure \ref{fig:optimal_route} visualizes this idea using the same expanded graph from Figure \ref{fig:expanded_graph}. Assuming sufficient vehicle capacity, the optimal solution follows the minimum-cost paths to serve two customers and then returns via the same paths in reverse. The red path in the figure represents the optimal vehicle route in the expanded graph, ensuring that no node or edge is visited more than once. 

Now, we are ready to formulate an MILP for the VRP-SA. We emphasize that it is based on the $k$-layer expanded graph $G_e$ instead of the original road network $G$. To keep track of the timestamps in different vehicle routes, we adopt the three-index formulation \cite{toth2014vehicle}. Define a binary decision variable $x_{ijm}$ for every $(i, j) \in E_e$ and every $m \in M$, with $x_{ijm} = 1$ indicating vehicle $m$ traverses edge $(i, j)$ in the expanded graph $G_e$. Let $\bm{x}$ be the vector of all $x_{ijm}$. Define a binary decision variable $y_{om}$ for every $m \in M$, with $y_{om} = 1$ indicating vehicle $m$ is dispatched. Let $\bm{y}_o$ be the vector of all $y_{om}$. The objective is to minimize the total operational cost, which includes both the routing cost, $RC(\bm{x})$, and the fixed cost of vehicle dispatching, $FC(\bm{y}_o)$. 
\begin{align}
    \min\limits_{\bm{x}, \bm{y}_o} \quad RC(\bm{x}) + FC(\bm{y}_o)
\end{align}

Specifically, the routing cost in Equation \ref{eq:rc} consists of three terms to account for the cost structure described in Table \ref{tab:cost}. The fixed cost in Equation \ref{eq:fc} depends on the number of dispatched vehicles of each type (i.e. AVs and HDVs).
\begin{align}
    RC(\bm{x}) = \sum_{(i, j) \in E, \; m \in M^h} c_{ij}^0 x_{ijm} + \sum_{(i, j) \in E^a_e, \; m \in M^a} c_{ij}^1 x_{ijm} + \sum_{(i, j) \in E^o_e, \; m \in M^a} c_{ij}^2 x_{ijm}, \label{eq:rc}
\end{align}
\begin{align}
    FC(\bm{y}_o) = f^a \sum_{m \in M^a} y_{om} + f^h \sum_{m \in M^h} y_{om}. \label{eq:fc}
\end{align}

Constraints \ref{c1} and \ref{c2} reflect the law of flow conservation for vehicle dispatching. Constraint \ref{c1} states that, for each vehicle and each node except the source depot $o$ and the sink depot $s$, the amount of incoming flow must be equal to the amount of outgoing flow.  Constraint \ref{c2} states that every vehicle, if dispatched from the source, must enter the sink ultimately. The flow value is capped by one unit to ensures that a vehicle does not return to the same node or edge in the expanded graph.
\begin{align}
    \sum_{j \in \delta^-(i)} x_{ijm} &= \sum_{j \in \delta^+(i)} x_{jim}, \qquad \forall \; i \in V_e - \{o, \; s\}, \; \forall \; m \in M, \label{c1} \\
    \sum_{j \in \delta^-(o)} x_{ojm} &= \sum_{j \in \delta^+(s)} x_{jsm} \leq 1, \qquad \forall \; m \in M. \label{c2}
\end{align}

Define a binary decision variable $y_{im}$ for every $i \in D_e$ and every $m \in M$, with $y_{im} = 1$ indicating vehicle $m$ serves a duplicate of customer $i$ in the expanded graph. Together with the previously defined decision variables $\bm{y}_o$, we interpret $y_{im} = 1$ for all $i \in D_e \cup \{o\}$ as an indicator of whether a vehicle serves a required node, either a customer or the depot.

Constraints \ref{c3} and \ref{c:capacity} are typical in other VRP formulations but have been adapted in our MILP to ensure compatibility with the expanded graph. Constraint \ref{c3} states that each customer is served by exactly one vehicle in total, though it may be served on any layer in the expanded graph. Constraint \ref{c:capacity} enforces that the total demand assigned to a vehicle does not exceed its capacity. 
\begin{gather}
    \sum_{i \in D^V(j)} \sum_{m \in M} y_{im} = 1, \qquad \forall \; j \in D, \label{c3} \\
    \sum_{i \in D_e} y_{im}\, d_i \leq W, \qquad \forall \; m \in M. \label{c:capacity}
\end{gather}

Constraint \ref{c4} ensures that a vehicle can visit a customer node without necessarily serving it. This situation arises when this customer lies on a vital path for a vehicle to reach another customer. In another word, a pass-through visit is allowed (i.e., $y_{im} = 0$ but  $\sum_{j \in \delta^+(i)} x_{jim} = 1$).
\begin{align}
    \sum_{j \in \delta^+(i)} x_{jim} \geq y_{im}, \qquad \forall \; i \in D_e, \; \forall \; m \in M. \label{c4}
\end{align}

Constraint \ref{c5} is called the transition constraint that guarantees a vehicle has to move to the upper layer after serving one customer. If vehicle $m$ does not serve a customer (i.e., $y_{im} = 0$), though it might have passed it, the vehicle stays on the same layer, which means $x_{ijm} = 0$ for dummy edge $(i, j) \in A$. Otherwise, the vehicle has to move up one layer, which means $x_{ijm} = 1$. 
\begin{equation}
    x_{ijm} = y_{im}, \qquad \forall \; (i, j) \in A, \; \forall \; m \in M. \label{c5}
\end{equation}

Then, we need a valid \textit{sub-tour elimination constraint} (SEC) to prevent a feasible route from being disconnected to the depot. In the context of VRP-SA, it is natural to devise the SEC in the Miller-Tucker-Zemlin (MTZ) form \cite{miller1960integer} when continuous variables (e.g., the timestamps) are necessary for the formulation. We define a continuous decision variable $t_{im} \in [0, T]$ for every $i \in V_e$ and every $m \in M$ to represent the timestamp of vehicle $m$ visiting node $i$. In particular, $t_{om}$ represents the departure time of vehicle $m$ in real-world operation.

Constraints \ref{c6} and \ref{c7} correspond to the SEC in the MTZ form. Constraint \ref{c6} states that the timestamp of a vehicle visiting a node should be strictly larger than any timestamps of the same vehicle visiting the predecessors. It prevents identical timestamps of different nodes in the same route, which eliminates sub-tours. The existence of Constraint \ref{c7} makes the inequalities in Constraint \ref{c6} become equalities if they are associated with a vehicle route (i.e., $x_{ijm} = 1$), which ensures that all timestamps in a route are consistent with respect to the travel time. Additionally, the operational end time is set as an upper bound for $t_{sm}$ to prevent any late return of vehicles.
\begin{align}
    &t_{jm} \geq t_{im} + \Delta t_{ij} + T \; (x_{ijm} - 1), \qquad \forall \; (i, j) \in E_e, \; m \in M, \label{c6} \\ 
    &T \geq t_{sm} = t_{om} + \sum_{(i,j) \in E_e} \Delta t_{ij} x_{ijm}, \qquad \forall \; m \in M. \label{c7}
\end{align}

Up to this point, we are able to keep track of all timestamps by decision variables. The next step is to formulate the budget constraint. We discretize the time horizon $[0, T]$ into a set of disjoint but connecting intervals $[a_q, b_q], \; \forall \; q \in Q$, where $Q$ is the set of interval indices. Mathematically, $b_q = a_{q+1}$ holds for $q = 1, \cdots, |Q|-1$. In practice, the intervals are set to be short (e.g., one minute) and can be equal. Then, we restrict the total number of occupied remote controllers, which is equivalent to the total number of AVs on ordinary roads simultaneously, to be less than the budget during any time interval $q$. 

For each $q \in Q$ and each $m \in M^a$, we define a binary decision variable $u_{qm}$, where $u_{qm} = 1$ indicates that vehicle $m$ is in need of a remote controller during the $q$-th time interval. Note that we only consider AVs since HDVs are not subject to the budget constraint. Then, $u_{qm} = 1$ if vehicle $m$ is traveling along edge $(i, j) \in E^o_e$ during $[t_{im}, t_{jm}]$ and $[a_q, b_q]$ overlaps with $[t_{im}, t_{jm}]$. The overlapping relation can be described by two conditions:
\begin{enumerate}
    \item The timestamp when vehicle $m$ enters edge $(i, j)$ has to be earlier than the end of the $q$-th time interval.
    \item The timestamp when vehicle $m$ leaves edge $(i, j)$ has to be later than the start of the $q$-th time interval.
\end{enumerate}
Figure \ref{fig:relation} illustrates such an overlapping relation. Suppose $[a_{10}, b_{10}]$ is the time interval of interest. $[t_1, t_2]$ and $[t_5, t_6]$ overlap with $[a_{10}, b_{10}]$ because both of the aforementioned conditions hold. $[t_8, t_9]$ does not overlap with $[a_{10}, b_{10}]$ because $t_8 > b_{10}$.
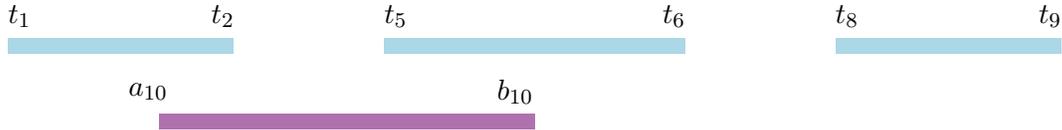
\begin{figure}[!htbp]
    \centering
    \begin{tikzpicture}

        
        \fill [vertexfill] (-7, 2.2) rectangle (-4, 2);
        \fill [vertexfill] (-2, 2.2) rectangle (2, 2);
        \fill [vertexfill] (4, 2.2) rectangle (7, 2);
        
        \fill [Orchid] (-5, 1.2) rectangle (0, 1);

        \node[text width=1cm] at (-6.5,2.5) {$t_1$};
        \node[text width=1cm] at (-3.8,2.5) {$t_2$};
        \node[text width=1cm] at (-1.5,2.5) {$t_5$};
        \node[text width=1cm] at (2.2,2.5) {$t_6$};
        \node[text width=1cm] at (4.5,2.5) {$t_8$};
        \node[text width=1cm] at (7.2,2.5) {$t_9$};
        
        \node[text width=1cm] at (-4.9,1.5) {$a_{10}$};
        \node[text width=1cm] at (0,1.5) {$b_{10}$};
    \end{tikzpicture}
    \caption{The relation of the overlapping relation between $[a_q, b_q]$ and $[t_{im}, t_{jm}]$}
    \label{fig:relation}
\end{figure}
More formally, $u_{qm} = 1$ if the following three conditions hold: (1) $x_{ijm} = 1$, (2) $a_q \leq t_{jm}$, and (3) $t_{im} \leq b_q$. Define a binary decision variable $\alpha_{qijm}$ for every time interval $q \in Q$, every edge $(i, j) \in E^o_e$, and every AV $m \in M^a$ to indicate whether $a_q \leq t_{jm}$. We have $a_q \leq t_{jm}$ if only if $\alpha_{qijm} = 1$. By analogy, we define $\beta_{qijm}$ to indicate whether $t_{im} \leq b_q$.

To capture such sufficient and necessary conditions, we can formulate Constraints \ref{c8} to \ref{c11} for every time interval $q \in Q$ and every vehicle-edge tuple $(m, \; (i,j)) \in S = M^a \times E_e^o$. 
\begin{align}
    &t_{jm} \leq a_q + T\, \alpha_{qijm}, \label{c8}\\
    &t_{jm} \geq a_q + T\, (\alpha_{qijm} - 1), \label{c9} \\
    &b_q \leq t_{im} + T\, \beta_{qijm}, \label{c10} \\
    &b_q \geq  t_{im} + T\, (\beta_{qijm} - 1). \label{c11}
\end{align}

Next, Constraint \ref{c12} is formulated to describe the sufficient condition for $u_{qm} = 1$. 
\begin{equation}
    u_{qm} \geq \frac{1}{3} (\alpha_{qijm} + \beta_{qijm} +x_{ijm}) - \frac{2}{3}, \qquad \forall \; q \in Q, \; \forall \; (m, (i,j)) \in S. \label{c12}
\end{equation}

The last but not least, we count how many remote controllers are occupied by the AV fleet during every time interval and cap it by the budget, which leads to Constraint \ref{c13}.
\begin{equation}
    \sum_{m \in M^a} u_{qm} \leq B, \qquad \forall \; q \in Q. \label{c13}
\end{equation}

The MILP is of polynomial size with respect to the size of the expanded graph and the number of discrete time intervals. The total number of decision variables and the total number of constraints are both $\mathcal{O}(k \; |S| \; |Q|) = \mathcal{O}(k \; |E_e^o| \; |M^a| \; |Q|)$. To summarize, the complete MILP is shown as follows.

\begin{mini}|s|[2]<b>
    {\bm{x}, \bm{y}, \bm{t}, \bm{\alpha}, \bm{\beta}, \bm{u}}
    {RC(\bm{x}) + FC(\bm{y}_o)}{}{}
    \addConstraint{\sum\limits_{j \in \delta^-(i)} x_{ijm}}{ = \sum\limits_{j \in \delta^+(i)} x_{jim} \leq 1, \qquad}{\forall \; i \in V_e - \{o, \; s\}, \; \forall \; m \in M}
    \addConstraint{\sum_{j \in \delta^-(o)} x_{ojm}}{= \sum\limits_{j \in \delta^+(s)} x_{jsm}, \qquad}{\forall \; m \in M}
    \addConstraint{\sum_{i \in D^V(j)} \sum\limits_{m \in M} y_{im}}{= 1, \qquad}{\forall \; j \in D}
    \addConstraint{\sum_{i \in D_e} y_{im} \, d_i}{\leq W, \qquad}{\forall \; m \in M}
    \addConstraint{\sum_{j \in \delta^+(i)} x_{jim}}{\geq y_{im}, \qquad}{\forall \; i \in D_e, \; \forall \; m \in M}
    \addConstraint{X_{ijm}}{ = y_{im}, \qquad}{\forall \; (i, j) \in A, \; \forall \; m \in M}
    \addConstraint{t_{jm}}{\geq t_{im} + \Delta t_{ij} + T \; (x_{ijm} - 1), \qquad}{\forall \; (i, j) \in E_e, \; \forall \; m \in M}
    \addConstraint{T}{\geq t_{sm} = t_{om} + \sum_{(i,j) \in E_e} \Delta t_{ij} x_{ijm}, \qquad}{\forall \; m \in M} 
    \addConstraint{t_{jm}}{\leq a_q + T\, \alpha_{qijm}, \qquad}{\forall \; q \in Q, \; \forall \; (m, (i,j)) \in S}
    \addConstraint{t_{jm}}{\geq a_q + T\, (\alpha_{qijm} - 1), \qquad}{\forall \; q \in Q, \; \forall \; (m, (i,j)) \in S}
    \addConstraint{b_q }{\leq t_{im} + T\, \beta_{qijm}, \qquad}{\forall \; q \in Q, \; \forall \; (m, (i,j)) \in S}
    \addConstraint{b_q }{\geq t_{im} + T\, (\beta_{qijm} - 1), \qquad}{\forall \; q \in Q, \; \forall \; (m, (i,j)) \in S}
    \addConstraint{u_{qm}}{\geq \frac{1}{3} (\alpha_{qijm} + \beta_{qijm} +x_{ijm}) - \frac{2}{3}, \qquad}{\forall \; q \in Q, \; \forall \; (m, (i,j)) \in S}
    \addConstraint{\sum_{m \in M^a} u_{qm}}{\leq B, \qquad}{\forall \; q \in Q}.
\end{mini}

\section{Extended Studies of the MILP} \label{sec:milp_extend}
The MILP proposed in Section \ref{sec:milp}, while exact and polynomial in size, faces limitations when applied to real-world VRP-SA instances. One challenge arises from the discretization of the operational time horizon, which creates a difficult trade-off between model size and performance. Fine granularity leads to a substantial increase in decision variables and constraints, whereas coarse granularity results in holding remote controllers longer than necessary, thereby reducing system capacity. Furthermore, as the system scales, the MILP may quickly become computationally intractable as it is much more complex than typical VRP formulations. In this section, we further explore the MILP formulation from various perspectives, aiming to enhance the model's comprehensiveness and adaptability. The analysis covered will guide the design of a two-phase tractable algorithm discussed in Section \ref{sec:algo}.

\subsection{The Budget Constraint from the Resource Allocation Perspective} \label{resource_allocation}
In the MILP, the time horizon of operation is discretized such that the budget constraint can be enforced at any given moment. In practice, however, the benefit of deploying the AV fleet deteriorates as the granularity of discretization increases, because a time interval may have a long non-overlapping sub-interval (e.g., $[t_2, t_5]$ in Figure \ref{fig:relation}) during which an AV loses the chance to utilize the remote controllers that are indeed available. Therefore, each time interval needs to be short enough. For example, assume the length of time intervals is set to be 30 seconds given that it is a reasonable time for an AV to switch from the automated mode to the remotely-controlled mode, the total number of time intervals in the MILP is as large as 1440 when we have a 12-hour operational window (i.e., $T = 12h$). According to Constraints \ref{c8} to \ref{c13}, the number of variables and constraints in the MILP is multiplied by at least $10^3$ to guarantee a relatively precise time discretization. In this section, we formulate an alternative set of constraints from the perspective of \textit{resource allocation} to meet the real-time budget without time discretization.

We override some notations specified in Table \ref{tab:notations} for a concise representation. Denote by $B$ the set of remote controllers to be allocated. For example, $B$ can be a quad of remote engineers to assist the AV fleet when AVs are on ordinary roads. Let $w = (w^1, w^2) \in E_e^o$ be an AV-enabled edge in the expanded graph, with $w^1$ and $w^2$ representing two end nodes. Define a binary decision variable $u_{bmw}$ for every remote controller $b \in B$, every AV $m \in M^a$, and every edge $w = (i, j) \in E_e^o$. If $u_{bmw} = 1$, the $b$-th unit is assigned to vehicle $m$ while it is on edge $w$. The budget is implicitly satisfied as we only have $|B|$ remote controllers regardless of the assignment. Then, the key idea is to ensure a remote controller $b$ is not assigned to two vehicles simultaneously. More formally, for every two distinct vehicle-edge tuples $(m_1, w_1)$ and $(m_2, w_2)$ in the set $M^a \times E_e^o$ such that $m_1 \neq m_2$, we have $u_{bm_1w_1} + u_{bm_2w_2} \leq 1$ if the following conditions hold:
\begin{enumerate}
    \item Edge $w_1 = (w_1^1, w_1^2)$ is in the route of vehicle $m_1$, which means  $x_{m_1 w_1} = 1$.
    \item Edge $w_2 = (w_2^1, w_2^2)$ is in the route of vehicle $m_2$, which means $x_{m_2 w_2} = 1$.
    \item The time interval of vehicle $m_1$ traversing edge $w_1$ overlaps with the time interval of vehicle $m_2$ traversing edge $w_2$. 
\end{enumerate}
To describe condition 3 rigorously, we further define a binary decision variable $\alpha (w_1, w_2, m_1, m_2)$ for every pair of vehicle-edge tuples in the set $S^2 = \{ ((m_1, w_1), (m_2, w_2)) \; | \; m_1 \neq m_2, \; m_1, m_2 \in M^a, \; w_1, w_2 \in E^o_e\}$. If and only if $\alpha (w_1, w_2, m_1, m_2) = 1$,  the timestamp of vehicle $m_1$ entering edge $w_1$ is earlier than or equal to the timestamp of vehicle $m_2$ exiting edge $w_2$, which means $t_{w_1^1 m_1} \leq t_{w_2^2 m_2}$. By analogy, we define $\beta (w_1, w_2, m_1, m_2)$ to indicate whether the timestamp of vehicle $m_1$ exiting edge $w_1$ is later than or equal to the timestamp of vehicle $m_2$ entering edge $w_2$, namely $t_{w_2^1 m_2} \leq t_{w_1^2 m_1}$. To capture such sufficient and necessary conditions, we can formulate Constraints \ref{c14} to \ref{c17} for every element in $S^2$.
\begin{align}
    &t_{w_2^2 m_2} \leq t_{w_1^1 m_1}+ T\, \alpha (w_1, w_2, m_1, m_2), \label{c14} \\
    &t_{w_2^2 m_2} \geq t_{w_1^1 m_1} + T\, (\alpha (w_1, w_2, m_1, m_2) - 1), \label{c15} \\
    &t_{w_1^2 m_1} \leq t_{w_2^1 m_2} + T\, \beta(w_1, w_2, m_1, m_2), \label{c16} \\
    &t_{w_1^2 m_1} \geq  t_{w_2^1 m_2} + T\, (\beta (w_1, w_2, m_1, m_2) - 1). \label{c17}
\end{align}

Next, for every $b \in B$ and every element in $S^2$, we can formulate Constraint \ref{c18} to describe the sufficient condition for $u_{bm_1w_1} + u_{bm_2w_2} \leq 1$. Finally, the set of Constraints \ref{c14} to \ref{c18} is a replacement for Constraints \ref{c8} to \ref{c13} for a complete MILP. 
\begin{equation}
    u_{bm_1w_1} + u_{bm_2w_2} \leq \frac{11}{4} - \frac{1}{4} (x_{m_1 w_1} + x_{m_2 w_2} + \alpha (w_1, w_2, m_1, m_2) + \beta(w_1, w_2, m_1, m_2)). \label{c18}
\end{equation}

Discarding time discretization is not always beneficial from the complexity point of view. The number of decision variables and the number of constraints in the new MILP are $\mathcal{O}(k \; |S^2| \; |B|)$, which may be more complex than that of the first MILP when the input is a huge road network because $|S^2|$ is $\mathcal{O}(|S|^2)$. In practice, however, we cannot conclude which formulation is superior to the other unless a specific instance is given. 

\subsection{Size Reduction of the MILP} \label{size_reduction}
The MILP proposed in Section \ref{sec:milp} has three folds of extra complexity compared to a basic VRP formulation: (1) The vehicle routing is modeled in a $k$-layer expanded graph. (2) Each layer is a real-world road network instead of the metric closure constructed on the set of customers. (3) A set of constraints, similar to those found in job scheduling problems, is added to the routing component in order to meet the budget of available remote controllers for real-time AV operations. Hence, in this section, we discuss two pre-processing heuristics to reduce the size of the MILP while keeping the sub-optimality gap small in general.

\subsubsection*{Heuristic 1: Reducing the Number of Layers in the Expanded Graph}
In Proposition \ref{prop: k}, we concluded that the number of layers in the expanded graph has to be at least $k$ to guarantee the optimality of the MILP, where $k$ is the maximum number of customers a vehicle can potentially serve under any circumstances and the upper bound of $k$ is the total number of customers in the instance. However, in practice, most urban networks follow a grid-like structure and consist predominantly of two-way roads, which is very different than the pathological example used to prove Lemma \ref{lemma:2}. For more details, refer to Figure \ref{fig:example_network} in Appendix \ref{app: A}. Moreover, the customers and the depot are spatially distributed more uniformly, resulting in less intersecting minimum-cost paths between them. In many of these real-world instances, the chance of a vehicle visiting a single intersection very frequently is significantly reduced. In other words, an optimal vehicle route may pass a single intersection significantly fewer times than $k$. Example \ref{real} illustrates such a realistic instance.
\begin{example} \label{real}
An undirected road network is shown in Figure \ref{fig:real_exp1}. Every node represents an intersection and every edge represents a road segment. The sky-blue node represents the depot and the light-purple nodes represent customers. The routing cost on an edge is proportional to its length. Suppose only one vehicle with unlimited capacity is available in the system. The optimal route can be obtained by observation, which is highlighted in Figure \ref{fig:real_exp2}. Notice that there are 6 customers but only node 2 and node 8 are visited twice. 
\end{example}
\input{figures/example_heuristic}

Hence, it is justified to reduce the number of layers in the expanded graph from $k$ to a smaller integer $\overline{k}$. Nevertheless, a naive reduction may cause a highly sub-optimal solution or even make the MILP infeasible because the transition Constraint \ref{c5} forces every vehicle to move one layer up every time it serves a customer except that it is already on the topmost layer. When $\overline{k} << k$, any vehicle tends to move to the topmost layer as soon as possible even though the route does not intersect with itself in serving the first few customers. In Example \ref{real}, a 3-layer expanded graph is sufficient as no node is visited more than twice in the optimal solution shown in Figure \ref{fig:real_exp2}, but the existence of Constraint \ref{c5} excludes this true optimal solution from the feasible space, resulting in a sub-optimal solution that re-routes to edge $(5, 6)$ instead of edge $(5, 2)$.

To circumvent this disadvantage, we re-design the transition constraint by replacing Constraint \ref{c5} with Constraint \ref{transition}. Due to the inequality, a transition of layers after serving a customer is no longer compulsory but optional. The transition only occurs if it leads to a better solution. Despite the improvement by this soft transition constraint, an ``optimal'' solution given by the updated MILP can still deviate from a true optimum of the problem if $\overline{k}$ is too small for a given instance. One strategy is to solve the MILP with $\overline{k} = 1$ and iteratively increase $\overline{k}$ until the optimal value converges or the reduced cost is less than a small threshold.  
\begin{align}
    x_{ijm} \leq y_{im}, \qquad \forall \; (i, j) \in A, \; \forall \; m \in M. \label{transition}
\end{align}

\subsubsection*{Heuristic 2: Pre-Prune the Road Network}
The network pre-pruning heuristic is motivated by Proposition \ref{pruning} presented in the work by Rodriguez-Pereira et al. \cite{RODRIGUEZPEREIRA2019615}. Recall that the VRP-SA generalizes from the STSP, where both the depot and customers are regarded as required nodes. Intuitively, if an edge is part of an optimal solution to a VRP-SA instance, it is highly probable to appear in a minimum-cost path between customers or in a path that utilizes remote controllers for the least amount of time. Therefore, it is justifiable to pre-prune the input road network before we construct the expanded graph.

The pre-processing steps are as follows: (1) Compute the minimum-cost paths between any pair of nodes in the set of customers and the depot. Denote by $P_1$ and $V_1$ the set of all edges and nodes used in those paths, respectively. (2) Compute the paths for any pair of nodes in the set of customers and the depot such that each path utilizes remote controllers for the shortest duration. Denote by $P_2$ and $V_2$ the set of edges and nodes used in those paths, respectively. (3) Identify all other AV-enabled strongly connected components that are connected to $P_1$ and $P_2$. Denote by $P_3$ and $V_3$ the set of all edges and nodes used in such components, respectively. (4) Remove all edges that are not in $P_1 \cup P_2 \cup P_3$ and all nodes that are not in $V_1 \cup V_2 \cup V_3$. The resulting sparse subgraph $G_s = (V_1 \cup V_2 \cup V_3, \; P_1 \cup P_2 \cup P_3)$ replaces the original network $G$ and is used to construct the expanded graph $G_e$. 
\begin{proposition} \label{pruning}
In any optimal solution to a given Steiner traveling salesmen problem instance, all edges used belong to some minimum-cost paths in the network connecting two required nodes. 
\end{proposition}

\section{A Two-Phase Tractable Algorithm} \label{sec:algo}
As discussed in Section \ref{size_reduction}, the size complexity of the original MILP renders it computationally intractable for any practical instances in the real world. Even with the application of the aforementioned size reduction heuristics in Section \ref{size_reduction}, directly solving the VRP-SA using the most powerful MILP solver remains challenging. Therefore, in this section, we develop a two-phase algorithm to efficiently find high-quality solutions to medium-size VRP-SA instances with up to 100 customers and 2000 edges in the road network. 

In phase 1, given any instance $\mathcal{I}_0$ of the VRP-SA, we create an instance $\mathcal{I}_1$ of the \textit{Heterogeneous Vehicle Routing Problem with Fixed Costs and Vehicle-Dependent Routing Cost (H-VRP-FD)} by removing the budget constraint of remote controllers. Naturally, the optimal value of $\mathcal{I}_1$ is a lower bound for that of $\mathcal{I}_0$. Then, we solve $\mathcal{I}_1$ efficiently to obtain a near-optimal solution. The exact optimum can be also pursued if the instance size permits. Currently, the state-of-the-art exact and heuristic approaches for this VRP variant are due to Baldacci and Mingozzi \cite{baldacci2009unified} and Vidal et al. \cite{VIDAL2014658}, respectively.

By decoding the solution to $\mathcal{I}_1$, we can derive all routes of vehicles dispatched. Denote by $\Tilde{X}$ the set of routes of the HDV fleet, and by $X$ the set of routes of the AV fleet. In particular, $\Tilde{X} = \{X(m) \; | \; \forall \; m \in \overline{M}^h\}$ and $X = \{X(m) \; | \; \forall \; m \in \overline{M}^a\}$, where $X(m)$ is the \textit{route} of vehicle $m$ represented as a sequence of road intersections visited, and $\overline{M}^h$ ($\overline{M}^a$) is the set of all dispatched (HDVs) AVs. Each $X(m)$ can be used to infer the \textit{routing schedule} $\mathcal{T}(m)$, which contains the timestamps of the vehicle visiting intersections on the route. All vehicle departure times are set to be zero. Denote by $\Tilde{\mathcal{T}} = \{\mathcal{T}(m) \; | \; m \in \overline{M}^h\}$ the set of routing schedules of dispatched HDVs, and by $\mathcal{T} = \{\mathcal{T}(m) \; | \; m \in \overline{M}^a\}$ the set of routing schedules of dispatched AVs. 

In phase 2, we check whether the set of routing schedules $\mathcal{T}$ is feasible with respect to the budget of remote controllers, and recover the solution feasibility if necessary. Recall that $\Tilde{\mathcal{T}}$ does not impact the feasibility of the solution to $\mathcal{I}_0$ since HDVs do not require assistance from remote controllers. However, the schedules $\mathcal{T}$ for AVs might be infeasible. For any two consecutive timestamps, $t_{im}$ and $t_{jm}$, in the routing schedule $\mathcal{T}(m)$, we can determine whether AV $m$ occupies a remote controller depending on the type of edge $(i, j)$. This enables us to efficiently verify whether $(X, \; \mathcal{T})$ violates the budget constraint over certain time intervals. If the budget is not exceeded, then $(\Tilde{X} \cup X, \; \Tilde{\mathcal{T}} \cup \mathcal{T})$ is a feasible and therefore an optimal solution to $\mathcal{I}_0$. On the contrary, if the budget is exceeded, we solve a \textit{feasibility recovering sub-problem (FRP)} to find a near-optimal solution $(\Tilde{X} \cup X^*, \;  \Tilde{\mathcal{T}} \cup \mathcal{T}^*
)$ to $\mathcal{I}_0$, where $(X^*, \; \mathcal{T}^*)$ are the updated AV routes and their schedules. Two families of FRP based on \textit{re-scheduling} strategy and \textit{re-routing} strategy will be discussed in Section \ref{sec:frp1} and Section \ref{sec:frp2}, respectively.

The two-phase algorithmic pipeline is detailed in Algorithm \ref{algo:holistic}, with the sub-routines in Line 3 and Line 6 elaborated later. It is important to note that neither the re-scheduling FRP nor re-routing FRP guarantees a complete recovery of the solution feasibility for $(\mathcal{X}, \mathcal{T})$. In the worst case, it may be necessary to replace some AVs with HDVs to serve certain customers and avoid exceeding the budget. If no additional HDVs are available in the system, the instance will be deemed infeasible.
\begin{algorithm}[!htbp]
\caption{The Two-Phase Algorithmic Pipeline to Solve the VRP-SA}

\begin{algorithmic}[1]
\Require A VRP-SA instance $\mathcal{I}_0$
\Ensure Feasible vehicle routes are schedules $(\Tilde{X} \cup X^*, \; \Tilde{\mathcal{T}} \cup \mathcal{T}^*)$

\State Construct an instance $\mathcal{I}_1$ of the H-VRP-FD based on $\mathcal{I}_0$, 
\State $(\Tilde{X}, \; X, \; \Tilde{\mathcal{T}}, \; \mathcal{T}) \gets \texttt{HVRPPDSolver} (\mathcal{I}_1)$ 
\State $\mathcal{T}^* \gets \texttt{ReSchedulingFRPSolver}(\mathcal{T})$
\State $X^* \gets X$
\If{$\mathcal{T}^* = \varnothing$}
\State $(\overline{D}, \; X^*, \; \mathcal{T}^*) \gets \texttt{ReRoutingFRPSolver}(X, \; \mathcal{T})$
\If{ $\overline{D} \neq \varnothing$}
\State $\Tilde{X}_{\overline{D}} \gets \texttt{CVRPSolver}(\overline{D})$
\State $\Tilde{X} \gets \Tilde{X} \cup \Tilde{X}_{\overline{D}}$
\EndIf
\EndIf
\end{algorithmic}
 \label{algo:holistic}
\end{algorithm}

Theoretically, Algorithm \ref{algo:holistic} may still fail to return a feasible solution even though $\mathcal{I}_0$ is indeed feasible. This limitation stems from the inherent nature of the relax-then-recover framework, which decomposes customer assignment and sequencing without incorporating a mechanism to iteratively explore the entire search space of the VRP-SA. Consequently, it may miss feasible solutions that require more nuanced exploration. However, we hypothesize that, under realistic conditions in the real world, a VRP-SA instance is very likely to be infeasible if Algorithm \ref{algo:holistic} fails to find a feasible solution.

\subsection{The Re-Scheduling Feasibility Recovering Sub-Problem} \label{sec:frp1}
One strategy to recover the feasibility of $(X, \; \mathcal{T})$ is to re-schedule $\mathcal{T}$ while keeping the set of routes $X$ unchanged. For each routing schedule $\mathcal{T}(m)$, we can extract a subset of timestamps called \textit{transition timestamps}, at which an AV switches between an AV-enabled road and an ordinary road. All road segments of the same type traveled between any two consecutive transition timestamps can be merged into a \textit{sub-route}. We observe that, the departure time is not necessarily zero for every vehicle in the optimal solution to the VRP-SA, as the departure delay of some vehicles may benefit the system by adding more feasibility to the temporal space of the problem. In another word, the routing schedule of every AV, as a whole, can potentially be shifted forward in time to avoid budget violations, provided the vehicle’s return time is earlier than the end of the operation horizon $T$.

Figure \ref{fig:frp1} illustrates this re-scheduling process for a small instance with two AVs and a single remote controller. In the figure, each strip represents a time interval during which an AV occupies a remote controller on an ordinary sub-route. The original schedules are infeasible due to overlapping strips between vehicles $m_1$ and $m_2$. If we delay the departure of vehicle $m_2$ by $\delta$, highlighted in purple strips, the schedules become feasible. This rescheduling does not increase the overall cost, as a departure delay does not inherently affect the operation negatively. In general, the \textit{re-scheduling FRP} takes $\mathcal{T}$ as the input and aims to find a feasible shift of the routing schedule $\mathcal{T}(m)$ for every dispatched vehicle $m$, such that the updated schedules $\mathcal{T}^*$ are feasible to $\mathcal{I}_0$.
\input{figures/re-scheduling-visualization}

The re-scheduling FRP is a feasibility problem that can be exactly modeled by an MILP without an objective function, referred to as the \textit{re-scheduling MILP}, which will be elaborated upon later. Notably, the formulation of this re-scheduling MILP resembles the constraints in Section \ref{resource_allocation} which handled the budget constraint from a resource allocation perspective.

Denote by $R = R^a \cup R^o$ the set of sub-routes, where $R^a$ is the subset of AV-enabled sub-routes and $R^o$ is the subset of ordinary sub-routes. Denote by $\Delta t_r$ the travel time in sub-route $r$. Define a continuous decision variable $\overline{t}_m \in  [0, T]$ as the re-scheduled departure time for each dispatched AV $m \in \overline{M}^a$. Furthermore, define two continuous decision variables $t^1_r \in [0, T]$ and $t^2_r \in [0, T]$ as the timestamp of the associated vehicle entering and exiting sub-route $r$, respectively. Then, the consistency among timestamps can be captured by Constraints \ref{frp1:c1} and \ref{frp1:c2}, where $r' \prec r$ means sub-route $r'$ is traversed earlier than sub-route $r$ on the same route. 
\begin{align} 
    &t_r^1 = \overline{t}_m + \sum_{r': \; r' \prec \, r} \Delta t_{r'}, \qquad \forall \; r \in R, \label{frp1:c1} \\
    &t_r^2 = t_r^1 + \Delta t_r, \qquad \forall \; r \in R. \label{frp1:c2}
\end{align}

Define a binary variable $u_{br} \in \{0, 1\}$ for every remote controller $b \in B$ and every ordinary sub-route $r \in R^o$. If $u_{br} = 1$, the $b$-th controller is assigned to the sub-route $r$. Constraint \ref{frp1:extra} regulates the unique assignment of one remote controller to a single ordinary sub-route.
\begin{gather}
    \sum_{b \in B} u_{br} = 1, \qquad \forall \; r \in R^o.
    \label{frp1:extra}
\end{gather}

Define a set $R^2 = \{(r_1, r_2) \; | \; r_1, r_2 \in R^o, \; r_1 \nparallel r_2\}$, where $r_1 \nparallel r_2$ means sub-routes $r_1$ and $r_2$ are not on the same route. For each pair of sub-routes in $R^2$, define a binary decision variable $\alpha(r_1, r_2)$. If and only if $\alpha(r_1, r_2) = 1$, the start time of sub-route $r_1$ is earlier than or equal to the end time of sub-route $r_2$, which means $t_{r_1}^1 \leq t_{r_2}^2$. By analogy, we define $\beta (r_1, r_2)$ to indicate whether the the end time of sub-route $r_1$ is later than or equal to the start time of sub-route $r_2$, namely $t_{r_2}^1 \leq t_{r_1}^2$. To capture these sufficient and necessary conditions, we formulate Constraints \ref{frp1:c3} to \ref{frp1:c6}.
\begin{align}
    &t_{r_2}^2 \leq t_{r_1}^1 + T \; \alpha(r_1, r_2), \qquad \forall \; (r_1, r_2) \in R^2,\label{frp1:c3} \\
    &t_{r_2}^2 \geq t_{r_1}^1 + T \; (\alpha(r_1, r_2) - 1),  \qquad \forall \; (r_1, r_2) \in R^2, \label{frp1:c4} \\
    &t_{r_1}^2 \leq t_{r_2}^1 + T \; \beta(r_1, r_2), \qquad \forall \; (r_1, r_2) \in R^2, \label{frp1:c5} \\
    &t_{r_1}^2 \geq t_{r_2}^1 + T \; (\beta(r_1, r_2) - 1),  \qquad \forall \; (r_1, r_2) \in R^2. \label{frp1:c6}
\end{align}

Then, we can formulate Constraint \ref{frp1:c7} to avoid a remote controller being assigned to two ordinary sub-routes simultaneously. Overall, all the constraints of the re-scheduling MILP are listed from \ref{frp1:c1} to \ref{frp1:c7}.
\begin{gather}
    u_{br_1} + u_{br_2} \leq \frac{5}{2} - \frac{1}{2} (\alpha(r_1, r_2) + \beta(r_1, r_2)), \qquad \forall \; (r_1, r_2) \in R^2. \label{frp1:c7}
\end{gather}

One may notice that the re-scheduling FRP is also an NP-hard problem, and the core constraints from \ref{frp1:c3} to \ref{frp1:c7} in the re-scheduling MILP exhibit structural similarity to the highly complex constraints discussed in Section \ref{resource_allocation}. In practice, however, the re-scheduling FRP can be solved much more efficiently as the total number of sub-routes is much less than the total number of edges in the network. It is justified because AV-enabled roads tend to form a skeleton within the network as smart infrastructure develops, leading to less frequent transitions between AV-enabled road segments and ordinary road segments along a vehicle route. 

Building on the concept of flexible departure times aforementioned, we further explore a scenario where the travel times of AVs on each road segment is also flexible. This unlocks an even larger feasible space for both the VRP-SA and the re-scheduling FRP. Detailed discussions are provided in Appendix \ref{Appendix B}.

\subsection{The Re-Routing Feasibility Recovering Sub-Problem} \label{sec:frp2}
In this section, we discuss the re-routing strategy to recover the feasibility of $(X, \mathcal{T})$. Unlike the re-scheduling FRP, the re-routing FRP takes both $X$ and $\mathcal{T}$ as the input and aims to find a feasible set of vehicle routes $X^*$ along with the updated routing schedules $\mathcal{T}^*$. Importantly, the customer-to-vehicle assignment remains unchanged, with adjustments made only within individual vehicle routes.

The re-routing FRP is an optimization problem with the objective of minimizing the additional cost incurred by route perturbation. It generalizes the re-scheduling FRP, as perturbing a route includes the special case of adjusting only the departure time (or all timestamps under the flexible travel time assumption discussed in Appendix \ref{Appendix B}). Due to its complexity, finding an exact solution is challenging. Therefore, we develop an iterative algorithm based on the \textit{ruin-and-recreate} principle and a \textit{route priority heuristic} to solve it approximately. 

First of all, we convert $(X, \mathcal{T})$ to a priority queue, where the priority of a route $X(m)$ is determined by its value $v(m)$, as defined in Equation \ref{eq:heuristic}. Two cost adjustment factors $\eta_1 < 1$ and $\eta_2 > 1$ that are defined in Section \ref{sec:description} play a key role here. Let $c^1(m)$ and $c^2(m)$ represent the total routing costs of AV-enabled roads and ordinary roads, respectively, in route $X(m)$. Routes that better utilize AV-enabled roads are more valuable compared to HDV routes. We also initialize $X^{*}$ and $\mathcal{T}^{*}$ as two empty containers to store feasible routes and their routing schedules.
\begin{equation} \label{eq:heuristic}
    v(m) = \frac{1 - \eta_1}{\eta_1} \; c^1 (m) + \frac{1 - \eta_2}{\eta_2} \; c^2(m)
\end{equation}

Then, in each iteration, the algorithm pops the route $X(m)$ with the highest value from the priority queue and examines whether its routing schedule $\mathcal{T}(m)$, when combined with any other existing routing schedules in $\mathcal{T}^{*}$, violates the budget constraint at a certain point in time. If no violation occurs, $X(m)$ and $\mathcal{T}(m)$ are moved to the feasible containers, and the algorithm proceeds to the next iteration. If a violation is detected, the current route $X(m)$ is ``ruined'', and the algorithm attempts to recreate a new route $X^*(m)$ such that the following two constraints are satisfied: (1) Vehicle $m$ avoids ordinary road segments when the budget is exhausted or exceeded given the schedules of remaining routes. (2) The return time of vehicle $m$ is less than the end time of operation. The re-creation process is essentially solving a \textit{Constrained Steiner Traveling Salesman Problem} (CSTSP). It takes as input the original network $G$, the subset of customers $D(m)$ served by the current route $X(m)$, the operational end time $T$, and the feasible routing schedules $\mathcal{T}^{*}$, and output a new feasible route $X^*(m)$ or report that no solution exists. Whenever the algorithm fails to recreate a feasible route cheaper than dispatching an HDV, it adds $D(m)$ to the set of unserved customers $\overline{D}$. 

The CSTSP, a crucial component of the algorithm, can be modeled similarly to the VRP-SA discussed in Section \ref{sec:milp}, but with substantially reduced complexity. We first apply the size reduction heuristics from Section \ref{size_reduction} and construct an expanded graph $G_e(m)$ following the steps in Section \ref{sec:milp}. The resulting graph is considerably smaller than $G_e$ for the full VRP-SA because the graph pre-pruning removes more edges given that $|D(m)| < |D|$. We then formulate a similar MILP, called \textit{re-routing MILP}, based on $G_e(m)$. The primary modifications are highlighted as follows while further details are provided in Appendix \ref{Appendix C}.
\begin{enumerate}
    \item Since the problem involves only one vehicle route, which must be dispatched, we eliminate the fixed cost component from the objective function and remove the vehicle index from all decision variables.
    \item Manual time discretization is unnecessary. Instead, we construct the set of non-overlapping time intervals by partitioning the operational time horizon $[0, T]$ based on the timestamps in the set $\{(i, j) \; | \; i, j \in \mathcal{T}' \cup \{T\}\}$, where $\mathcal{T}' = \bigcup_{\mathcal{T}(m) \in \mathcal{T}^{*}} \mathcal{T}(m)$ represents a merged set of timestamps from the current feasible routing schedules $\mathcal{T}^{*}$. 
\end{enumerate}

\begin{algorithm}[!t]
\caption{The re-routing FRP solver}
\begin{algorithmic}[1]
\Require Graph $G$, AV routes and schedules $(X, \; \mathcal{T})$, End time of operation $T$
\Ensure Updated AV routes and schedules $(X^*, \; \mathcal{T}^*)$, Unserved customers $\overline{D}$

\State $(X, \; \mathcal{T}) \gets \texttt{PriorityQueueByValues}(X, \; \mathcal{T})$
\State Initialize $(X^*, \mathcal{T}^*) = (\varnothing, \varnothing)$

\While{$(X, \mathcal{T}) \neq \varnothing$}

\State $(X(m), \mathcal{T}(m)) \gets (X, \mathcal{T}).\texttt{dequeue()}$

\If{$\mathcal{T}^{*} \cup \{\mathcal{T}(m)\}$ does not violate budget at any time}
\State $X^* \gets X^* \cup \{X(m)\}$
\State $\mathcal{T}^* \gets \mathcal{T}^* \cup \{\mathcal{T}(m)\}$
\State \textbf{continue}
\EndIf

\State $D(m) \gets \texttt{ExtractCustomers}(X(m))$
\State $G_s(m) \gets \texttt{GraphPruner}(G, D(m))$
\State $G_e(m) \gets \texttt{GraphExpander}(G_s(m))$
\State $(X^*(m), \; \mathcal{T}^*(m)) \gets \texttt{ReRoutingMILPSolver}(G_e(m), \; D(m), \; \mathcal{T}^*, \; T)$
\State $UB \gets \texttt{TSPSolver}(D(m))$
\If{$X^*(m) = \varnothing$ or $X^*(m)$ is more expansive than $UB$}
\State $\overline{D} \gets \overline{D} \cup D(m)$
\Else
\State $X^* \gets X^* \cup \{X^*(m)\}$
\State $\mathcal{T}^* \gets \mathcal{T}^* \cup \{\mathcal{T}^*(m)\}$
\EndIf
\EndWhile
\end{algorithmic}
\label{algo:re-routing-frp}

\end{algorithm}

Overall, Algorithm \ref{algo:re-routing-frp} provides a complete pseudo-code for the re-routing FRP solver. As aforementioned, the order of iteration of routes is not arbitrary but depends on the route priority heuristic. Routes with higher values are less likely to be perturbed or replaced, which guides the algorithm towards solutions with lower total routing costs. Additionally, if the algorithm terminates with a non-empty set $\overline{D}$, it indicates that the feasibility of $(X, \mathcal{T})$ has not been fully recovered. This outcome could be due to one of the following reasons: (1) the re-routing FRP is infeasible with the current customer-to-vehicle assignments, or (2) the re-routing FRP is feasible but the algorithm fails to find a feasible solution. Nevertheless, the algorithm is devised to recover the feasibility to the greatest extent possible.

\section{Numerical Experiments} \label{sec:exp}
In this section, we construct a tailored set of instances for the VRP-SA and solve them using Algorithm \ref{algo:holistic} described in Section \ref{sec:algo}. The results (1) validate the effectiveness and efficiency of the proposed two-phase algorithm and (2) illustrate how variations in input parameters influence both the problem structure and the corresponding solutions.

\subsection{VRP-SA Instances} \label{instances}
We construct 23 VRP-SA instances, each based on a CVRP instance from the benchmark dataset P described by Augerat \cite{augerat1995computational}. The number of customers in these CVRP instances ranges from 16 to 101. Each CVRP instance provides the following inputs: the coordinates of the depot and customers in a 2D Euclidean space, customer demands, and a uniform vehicle capacity. These inputs are directly adopted for the corresponding VRP-SA instances.

To accommodate the additional attributes of the VRP-SA, we construct an underlying road network for each CVRP instance based on the coordinates of all nodes. The process is as follows: (1) Draw a rectangular bounding box such that the farthest nodes are positioned along its boundary. (2) Construct a primary grid by equally dividing the bounding box along both the x- and y-axes, with two hyperparameters, $g_x$ and $g_y$, determining the number of divisions in each direction. Unless otherwise specified, we set $g_x = g_y = 5$ in the following experiments. This grid serves as the skeleton of the underlying road network, where all road segments are AV-enabled. (3) Within each cell of the primary grid, we construct a local grid, where each customer within the cell is placed precisely at an intersection. All road segments within the local grids are ordinary roads. Figure \ref{fig:network} illustrates the underlying road network of instance P-n40-k5. The design aims to mimic real-world semi-autonomous environments, where infrastructure upgrades to support AVs are prioritized on primary avenues in cities.  
\begin{figure}[!htbp]
    \centering
    \includegraphics[width=0.8\linewidth]{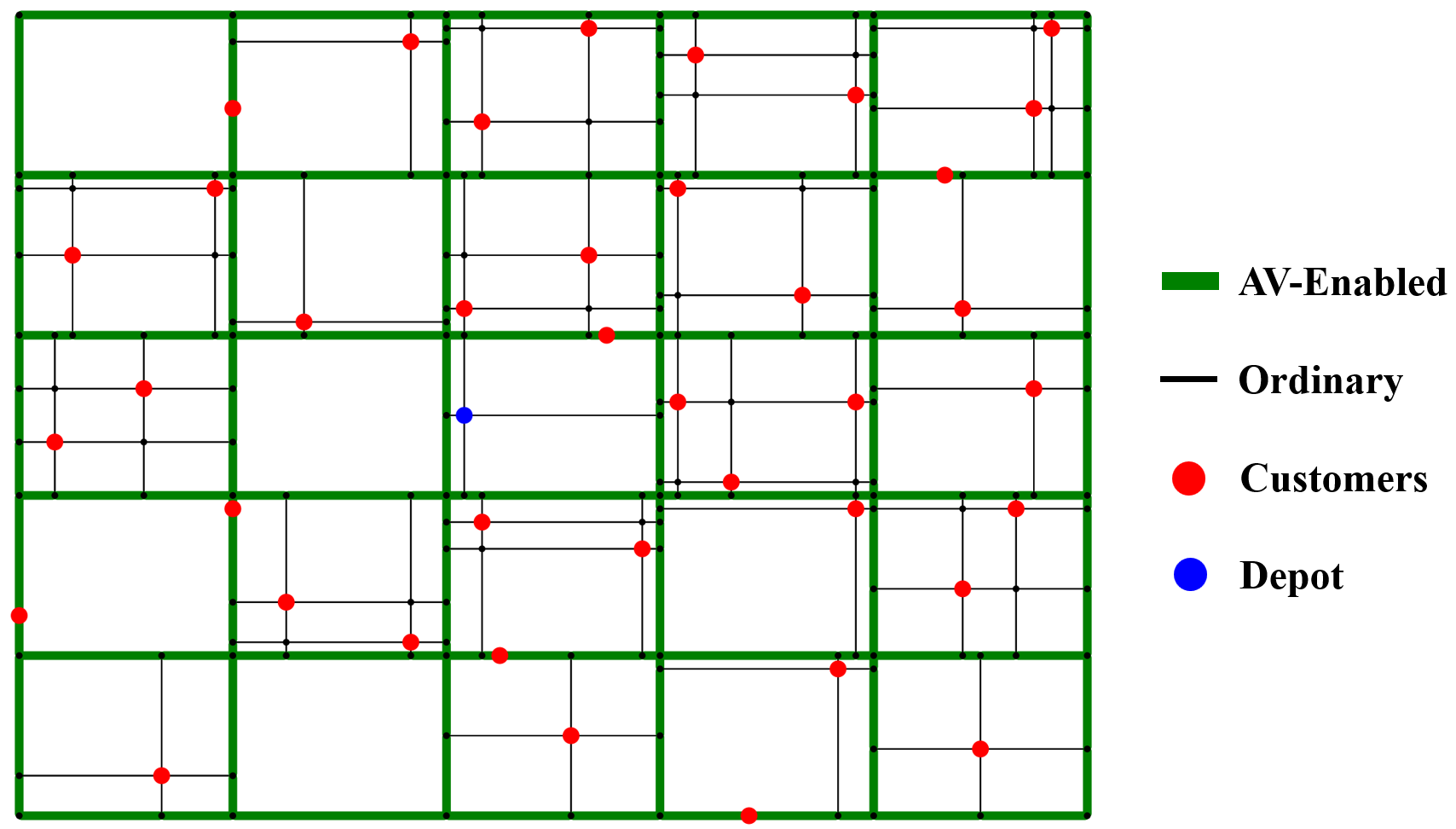}
    \caption{The underlying road network of P-n40-k5}
    \label{fig:network}
\end{figure}

As assumed in Section \ref{sec:description}, the unit routing cost of an AV on the primary grid is lower than that of an HDV. Conversely, it is more cost-demanding for an AV to travel on local grids than an HDV. For HDVs, the routing cost on any edge in the road network is equal to the edge length. For AVs, the routing cost on an edge in the primary grid is given by $\eta_1$ times the edge length, while on local grids, it is $\eta_2$ times the edge length. In the following experiments, we set $\eta_1 = 0.5$ and $\eta_2 = 1.2$ if not otherwise specified. 

To guarantee that all VRP-SA instances are feasible, which simplifies the analysis of experimental results, we set the size of the AV fleet, $|M^a|$, and the HDV fleet, $|M^h|$, in each VRP-SA instance to match the number of vehicles used in the optimal solution of the corresponding CVRP instance. To avoid over-dispatching, we assign an equal fixed dispatch cost (e.g., 1) to all vehicles.

For each VRP-SA instance, we define the operation end time as $T = T_{factor} \cdot \overline{T}$, where $\overline{T}$ is the maximum return time of all vehicles in the optimal solution to the CVRP instance with only HDVs deployed. The parameter $T_{factor} \in [1, \infty]$ serves as a time-horizon factor that controls the relative tightness of the operation time window. Note that travel time and routing cost are used interchangeably since the VRP instances in this section are synthetic. In the following experiments, we set $T_{factor} = 1.2$, unless otherwise specified.

Finally, we introduce a budget $B$ that determines the number of remote controllers available to assist AVs on ordinary local grids. In the experiments, we set $B = \max(1, \; \texttt{round}(\frac{|M^a|}{3}))$ unless otherwise specified, which roughly corresponds to one-third of the AV fleet size in each VRP-SA instance. This percentage-based approach provides more flexibility, as the minimum number of dispatched vehicles varies significantly across instances. 

\subsection{Algorithm Settings} \label{setting}
Strictly speaking, the algorithm must solve an H-VRP-PD in Phase 1, as shown in Line 2 of Algorithm \ref{algo:holistic}, before proceeding to Phase 2 for re-scheduling or re-routing. This initial step is computationally intensive. From an experimental standpoint, we simplify Phase 1 by solving a CVRP with only the AV fleet deployed, under the same settings. While this solution is suboptimal compared to that for the H-VRP-PD, it reduces computational demands. In practice, the algorithm’s performance can be further improved by exactly solving the H-VRP-PD, which would provide a true lower bound for the VRP-SA. For the CVRP, we use the state-of-the-art solver \textit{HGS-CVRP} developed by Vidal \cite{VIDAL2022105643}. The TSP solver used in Line 14 of Algorithm \ref{algo:re-routing-frp} is from Google OR-Tools. All other MILPs are solved using Gurobi 10.0.1, with several hyperparameters fine-tuned for each instance. Specifically, the \texttt{TimeLimit} for all Gurobi solvers is set to 300 seconds to enforce termination. Unless otherwise specified, the number of layers $k$ in the expanded graph is set to 2.

\subsection{Baseline Results} \label{result}
For each VRP-SA instance, we denote $f^P$ as the routing cost obtained using the proposed algorithm with the re-routing priority heuristic, and $f^{\overline{P}}$ as the routing cost obtained with a random re-routing order. Additionally, we denote $f_1$ as the near-optimal routing cost for the CVRP instance with only the AV fleet deployed, as obtained in Phase 1, and $f_2$ as the near-optimal routing cost for the CVRP instance with only the HDV fleet deployed, representing the cost without exploiting AVs at all. To account for variability due to instance sizes, we report the \textit{routing cost ratios} $f_2 / f_1$, $f^P / f_1$, and $f^{\overline{P}} / f_1$, instead of the absolute routing costs, for all instances ranging from 0 to 22. Given an instance, the quantity $f_2 / f_1 - 1$ approximately measures the percentage increase in routing cost when operating in a non-autonomous environment compared to a semi-autonomous environment with unlimited remote controllers. Similarly, $f^P / f_1 - 1$ or $f^{\overline{P}} / f_1 - 1$ quantify the percentage increase in routing cost under a limited budget of remote controllers.
\begin{figure}[!htbp]
    \centering
    \includegraphics[width = 0.9\linewidth]{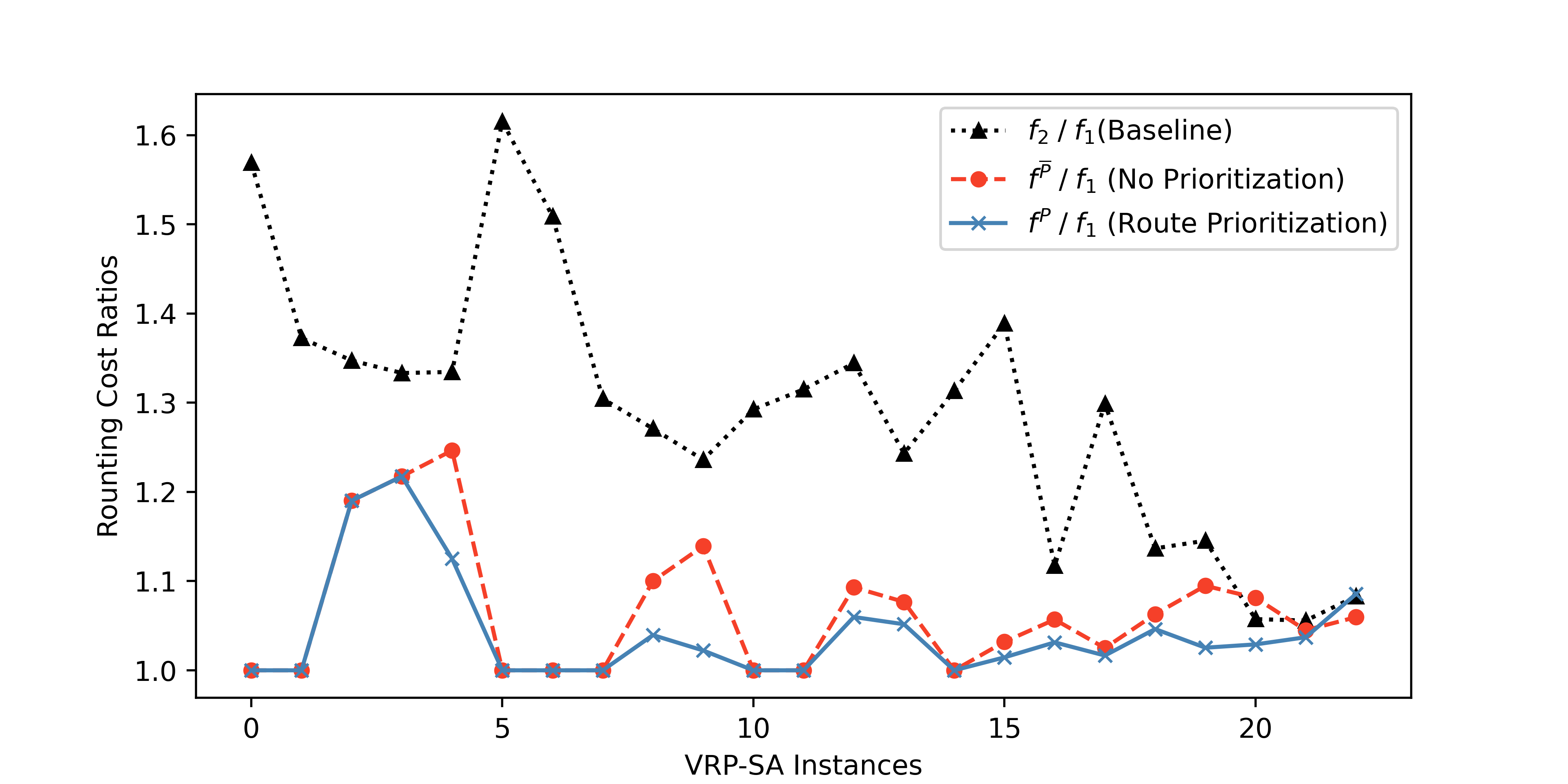}
    \caption{Routing cost ratios in different VRP-SA instances}
    \label{fig:main_result}
\end{figure}

As shown in Figure \ref{fig:main_result}, the proposed algorithm with only a random re-routing order generally yields lower routing cost ratios than the baseline with only the HDV fleet deployed (i.e., no AVs involved), except for instance 20. The greatest cost reduction is over 37.5\% in instance 5. If the re-routing priority heuristic is enabled, the algorithm performs even better in most cases, except for instance 22. Moreover, the algorithm successfully re-schedules the vehicle routes obtained in Phase 1 to feasibility without evoking the re-routing component in 8 out of 23 instances (i.e., instances 0, 1, 5, 6, 7, 10, 11, and 14), which is advantageous as it results in zero additional routing cost. More numerical details are available in Appendix \ref{Appendix C}.

We also observe two secondary patterns from the results: (1) As instance size increases, $f_2 / f_1$ tends to decrease. Since the instances in Figure \ref{fig:main_result} are ordered first by the number of customers and then by the number of available vehicles, this trend is expected. The number of divisions, $g_x$ and $g_y$, in the primary grid remains fixed across all instances, leading to a higher customer density within each grid cell for larger instances. Consequently, AV-enabled roads become relatively ``sparser'' in larger instances, diminishing the advantage of dispatching AVs. The effect of AV-enabled road density is further examined in Section \ref{heatmap}. (2) In instances where $f_2 / f_1$ is relatively low (e.g., instances 2, 3, 4, 8, 9, 13, and 16), $f^{\overline{P}} / f_1$ or $f^P / f_1$ tend to be relatively high, and vice versa. This pattern follows a similar rationale as the first. When $f_2 / f_1$ is low, AVs face greater difficulty in effectively leveraging AV-enabled roads, making re-routing less feasible and increasing its cost.

\subsection{Sensitivity Analysis of Input Parameters} \label{heatmap}
We conduct a series of experiments by varying the budgets and time-horizon factors while keeping other input parameters at their default values. Specifically, we set $B = \max(2, \; \texttt{round}(|M^a| \cdot B_{factor}))$ with $B_{factor} = (1/3, \; 1/2, \; 2/3)$ and $T_{factor} = (1.0, \; 1.1, \; 1.2, \; 1.3, \; 1.4, \; 1.5)$. The results shown in Figure \ref{fig:H1} and Figure \ref{fig:H2} are the average values across all 23 instances. As the budget and the tightness of the operational time horizon increase, the routing cost ratio $f^P / f_1$ decreases to as low as 1.006, resulting in a significant cost reduction of up to 22\%. Theoretically, an instance with a looser time-horizon factor should save more cost than that with a tighter time-horizon factor. The observed discrepancy for the results of $(B_{factor}, \; T_{factor}) = (1/2, \; 1.4)$ is attributed to the indeterministic behavior in the Gurobi solver under the parameter setting concerning the maximum running time.
\begin{figure}[!htbp]
    \centering
    \begin{subfigure}[t]{0.49\linewidth}
    \centering
    \includegraphics[width = \linewidth]{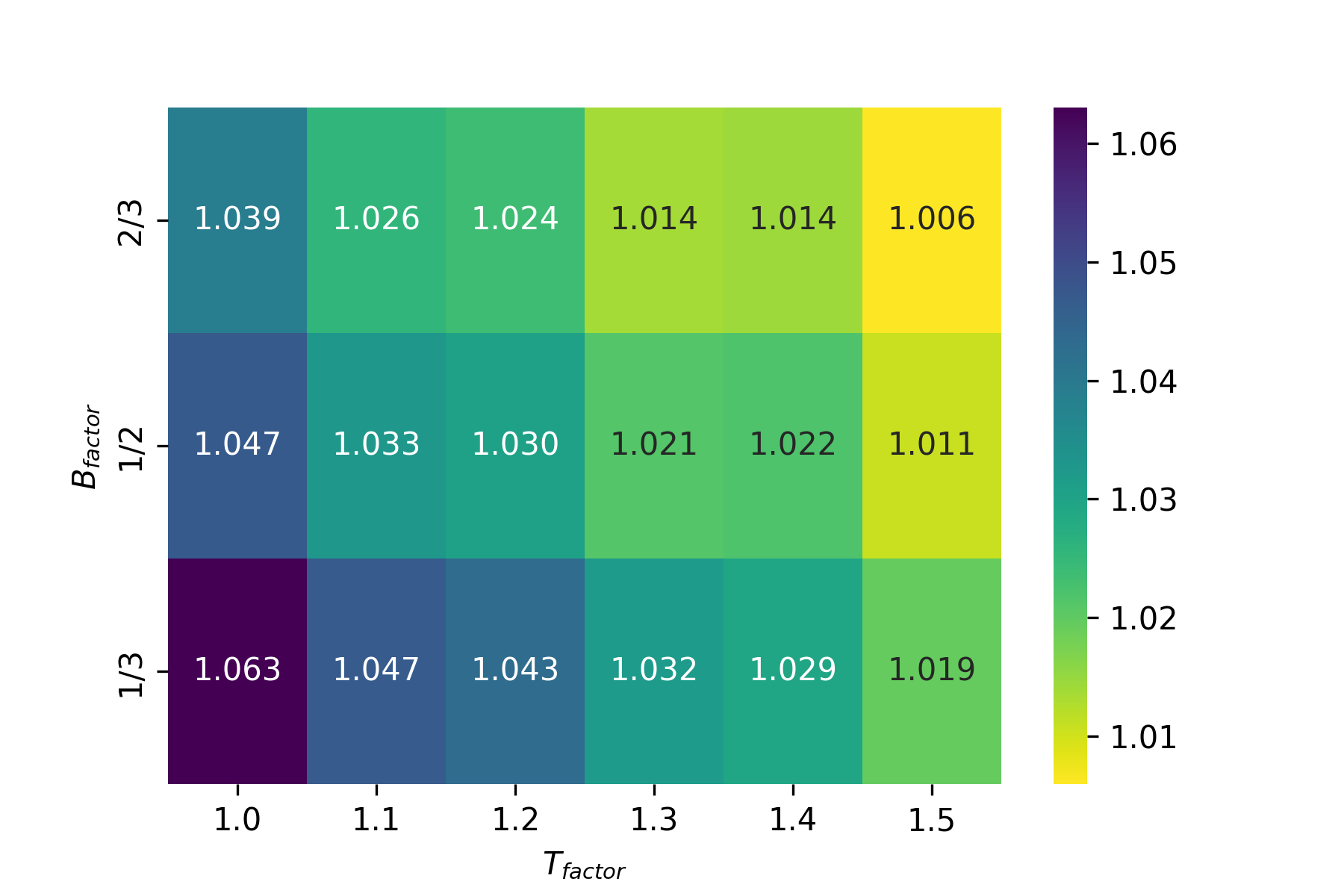}
    \caption{Average routing cost ratios $f^P / f^1$}
    \label{fig:H1}
    \end{subfigure}
    \begin{subfigure}[t]{0.49\linewidth}
    \centering
    \includegraphics[width = \linewidth]{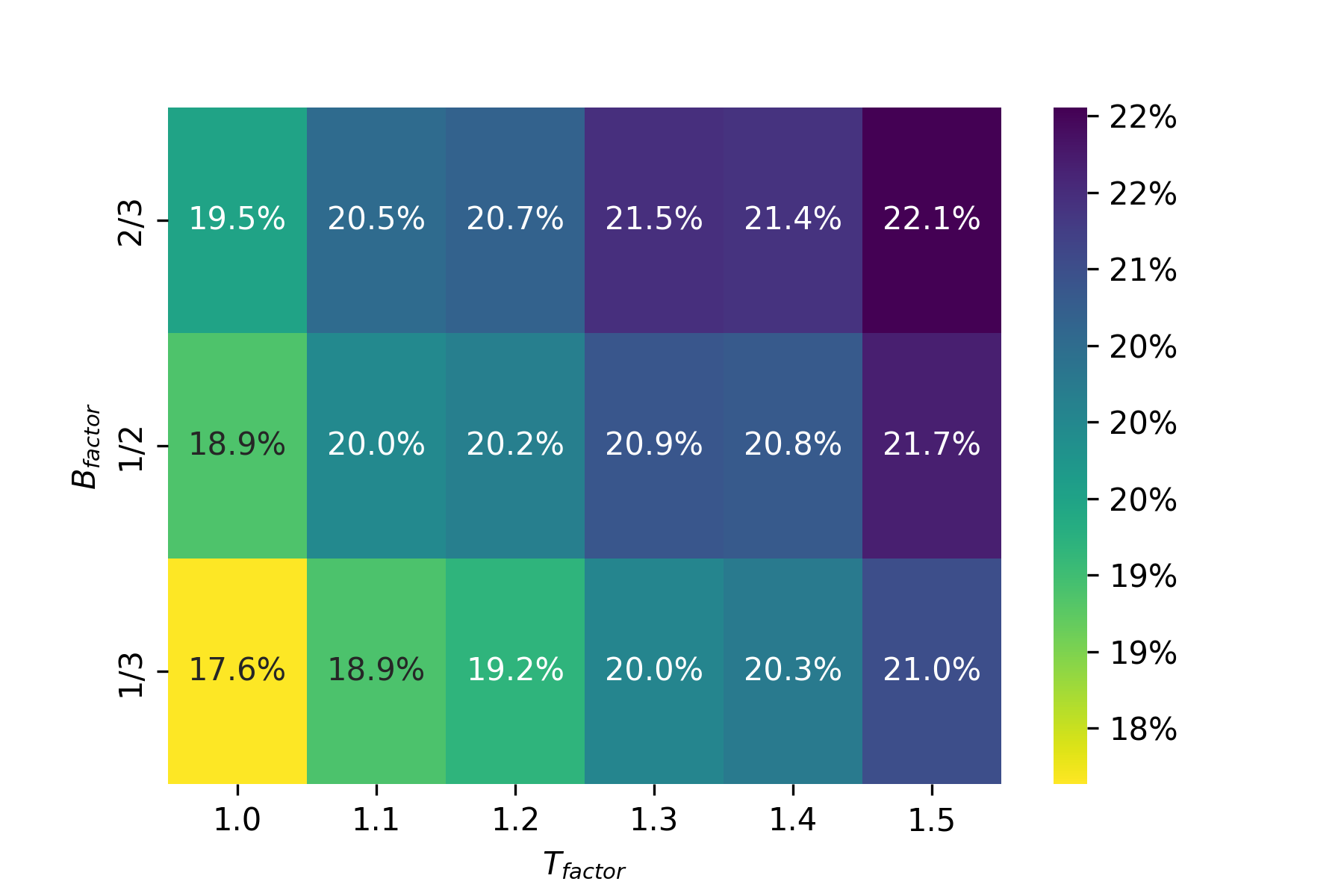}
    \caption{Average cost reduction due to VRP-SA}
    \label{fig:H2}
    \end{subfigure}
    \caption{Changing budget and the tightness of operational time horizon}
\end{figure}

We also vary the number of divisions $g_x$ and $g_y$ from 3 to 7 in the primary grid, while keeping other input parameters at their default values. Figure \ref{fig:grid_size} shows the \textit{inverse routing cost ratio} $f_1 / f_2$, where the ratio upper bound is 1.0, representing the cost of dispatching only HDVs. According to the result, it becomes increasingly advantageous to replace HDVs with AVs in networks with denser AV-enabled roads. Moreover, ratio $f^P / f_2$ closely follows the trend of $f_1/ f_2$, indicating that the proposed model and algorithm effectively leverage the increased density of AV-enabled roads to reduce routing costs. The density of AV-enabled roads serves as a metric for quantifying the degree of transition from a current environment to a fully autonomous one. Interestingly, the results reveal a two-stage pattern in this trend. In stage one (i.e., $g_x = g_y \leq 5$) when the density is relatively low, the gap between $f^P / f_1$ and $f_1 / f_2$ widens as the density increases. This is due to the budget constraint of remote controllers, which limits AVs from using all available road segments in routing. In stage two (i.e., $g_x = g_y \geq 5$) when the density surpasses a certain threshold, the gap starts to narrow as the density increases further. This is because the AV fleet relies less on remote control when most ordinary roads are upgraded to AV-enabled, thus diminishing the impact of the budget constraint on both re-scheduling and re-routing. This two-stage pattern reflects an inherited characteristic of the VRP-SA along the road network upgrade process. At the outset, dispatching AVs in routing tasks provides little or no benefit compared to HDVs. In the end, dispatching AVs is most beneficial but the problem structure is reduced back to a CVRP. Hence, high-quality solutions to the VRP-SA are the most valuable and yet the most difficult to obtain during the intermediate phase of semi-autonomous driving.
\begin{figure}[!htbp]
    \centering
    \includegraphics[width = 0.9\linewidth]{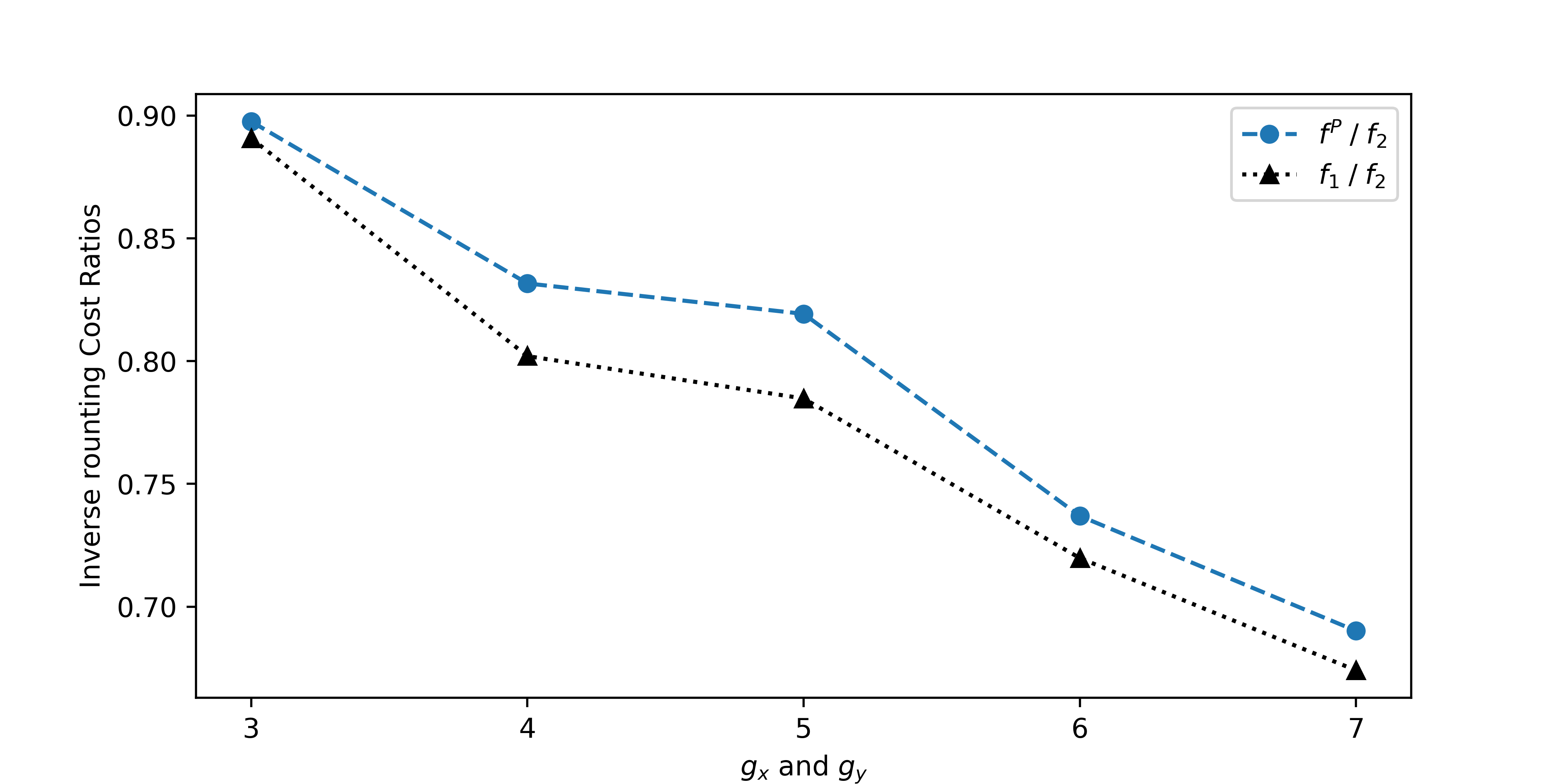}
    \caption{Average inverse routing cost ratios under various densities of AV-enabled roads}
    \label{fig:grid_size}
\end{figure}

\section{Conclusion and Future Directions} \label{sec:conclusion}
To conclude this paper, we introduce a novel variant of the vehicle routing problem, termed VRP-SA, designed for the semi-autonomous environment anticipated in the near future. We provide both a rigorous MILP model and an efficient solution approach. Extensive numerical experiments using a benchmark VRP dataset demonstrate the potential and benefits of deploying AV fleets for routing tasks, even with limited infrastructure support.

This paper represents a pioneering effort in exploring the impact of a semi-autonomous environment on a broad class of operations research problems related to fleet management and transportation networks. Although our focus is specifically on vehicle routing, we believe our findings can inspire further research in related areas. For example, a relevant question might be: Which subset of AV-enabled roads should be upgraded first in a city to support AV deployment? This could lead to novel network design problems that are not yet characterized by existing frameworks.

Regarding the VRP-SA itself, there is potential for further improvement in solution quality. The two-phase algorithm currently decomposes customer assignment and sequencing, limiting its scope to local search.Future research could investigate more advanced algorithms that enable a more comprehensive exploration of the solution space. Potential approaches include Lagrangian relaxation and Benders decomposition, which can exploit the special structures of the original MILP, as well as column generation techniques based on the set partitioning formulation. 

\printbibliography
\newpage
 
\appendix

\section{Proofs of Lemma \ref{lemma:1}, Lemma \ref{lemma:2}, and Proposition \ref{prop: k}} \label{app: A}

\lemmanode*
\begin{proof}
Suppose that there exists an optimal route $P^{*}$ for a given STSP instance. Denote by $u$ and $v$ two consecutive required nodes in $P^{*}$ and by $P(u, v) \subseteq P^{*}$ the sub-path between $u$ and $v$. Then, every node $x \in P(u, v)$ is visited only once in the scope of $P(u, v)$ as $P(u, v)$ is the minimum-cost path between $u$ and $v$ and there is no positive cycle in any minimum-cost paths. Then, if $x$ is not the depot, the number of times $x$ can be visited in the optimal route $P^{*}$ is upper bounded by $n - 1$ as it can potentially appear in the sub-paths between any pairs of consecutive required nodes. If $x$ happens to be the depot, this upper bound is $n$. Furthermore, the example depicted in Figure \ref{fig:star} demonstrates that the upper bound is tight. In a star network, all six nodes are required nodes with node 0 specified as the depot. Any sequences visiting nodes 1 to 5 exactly once are optimal. In all the optimal solutions to the STSP, node 0 is visited 6 times. The pattern can be found in a more general network where all required nodes are located in different isolated regions that can only be accessed via a single gateway. 
\end{proof}

\lemmaedge*
\begin{proof}
The formal proof of Lemma \ref{lemma:2} is similar to that of Lemma \ref{lemma:1} except that the tight upper bound is demonstrated by another example depicted in Figure \ref{fig:one-way}. In a one-way-dominated network shown in Figure \ref{fig:one-way}, nodes 0, 6, 8, and 10 are required nodes with node 0 specified as the depot. An optimal solution to the STSP is (0, 1, 2, 6, 5, 1, 2, 3, 8, 7, 5, 1, 2, 3, 4, 10, 9, 7, 5, 1, 0). Node 1 is visited 4 times and edge (1, 2) is traversed 3 times. The pattern can be found in a more general network where all required nodes are located in different isolated one-way-dominated regions that can only be accessed via an inevitable passage.
\end{proof}

\input{figures/example_stsp}

\propositionk*
\begin{proof}
Suppose $D(m)$ is the set of customers served by vehicle $m$ in the optimal solution and $o$ is the depot. Then, $P(m)$ is the optimal solution to the STSP in the same graph with $D(m) \cup \{o\}$ as the set of required nodes. Based on Lemma \ref{lemma:2}, the upper bound of the number of times $v$ being visited and $(u, v)$ being visited are $|D(m)| + 1$ and $|D(m)|$ respectively. By definition, $k$ is the maximum number of customers vehicle $m$ can potentially serve under any circumstances, meaning that $|D(m)| \leq k$. Thus, the proposition holds.
\end{proof}

\section{Flexible Travel Time for Autonomous Vehicles} \label{Appendix B}
In this section, we further assume an AV is allowed to increase or decrease its speed to some extent when it is cruising along a road segment $(i, j)$. Then, an AV can flexibly adjust the travel time on $(i, j)$ as long as the adjustment is within a threshold, which further enlarges the temporal feasible space of the problem. Define $\underline{\gamma}_{ij}$ ($\overline{\gamma}_{ij}$) as the \textit{minimum (maximum) adjustment factor} for the travel time on road segment $(i, j)$. Then, $[\underline{\gamma}_{ij} \Delta t_{ij}, \;  \overline{\gamma}_{ij} \Delta t_{ij}]$ is the corresponding flexible travel time interval. In this setting, Constraint \ref{c7} in the MILP in Section \ref{sec:milp} is removed and Constraint \ref{c6} is generalized to Constraints \ref{lower} and \ref{upper} for every AV $m \in M^a$. Note that the constraints corresponding to HDVs are kept intact. 
\begin{gather}
    t_{jm} \geq t_{im} + \underline{\gamma}_{ij} \Delta t_{ijm} + T \; (x_{ijm} - 1), \qquad \forall \; (i, j) \in E_e, \; m \in M^a, \label{lower} \\ 
    t_{jm} \leq t_{im} + \overline{\gamma}_{ij} \Delta t_{ijm} + T \; (1 - x_{ijm}), \qquad \forall \; (i, j) \in E_e, \; m \in M^a. \label{upper}
\end{gather}

One may argue that the flexible travel time assumption for the AV fleet is ill-conceived for two reasons: (1) It can impair network efficiency by blocking the traffic in some primary avenues. (2) It can affect the routing cost and change the objective of the problem. For this, we emphasize that it only generalizes the model and provides the possibility to design a better dispatching and routing strategy in the semi-autonomous environment. The adjustment factor of every road segment can be determined separately and designed according to the input network. Also, we admit a linear mapping between the routing cost and the travel time, which makes the problem objective easily adaptable to cases with flexible travel time.

\subsection{Effects on the Re-Scheduling MILP}
Following the notations used in Section \ref{sec:frp1}, when the flexible travel time assumption is adopted, the time interval of any sub-route $r$ can be independently adjusted by a factor of $\underline{\gamma}_{r}$ for shrinking or $\overline{\gamma}_{r}$ for expanding, where $\underline{\gamma}_r$ and $\overline{\gamma}_r$ are determined in Equations \ref{eq:B1-1} and \ref{eq:B1-2}:
\begin{align}
    \underline{\gamma}_r = \frac{1}{\Delta t_r} \sum_{(i, j) \in r} \underline{\gamma}_{ij} \Delta t_{ij} \label{eq:B1-1}\\
    \overline{\gamma}_r = \frac{1}{\Delta t_r} \sum_{(i, j) \in r} \overline{\gamma}_{ij} \Delta t_{ij} \label{eq:B1-2}
\end{align}

Figure \ref{fig:frp1-1} illustrates how this assumption results in a larger feasible space of the re-scheduling FRP. The original schedules are the same as the ones described in Figure \ref{fig:frp1}, except that the second ordinary sub-route of vehicle $m_1$ is longer. Consequently, the AV-enabled sub-route between the second and the third ordinary sub-routes becomes shorter. In this case, the re-scheduling FRP with fixed travel times is infeasible, as shifting the entire routing schedule of vehicle $m_2$ cannot bypass all overlaps. However, under the flexible travel time assumption, we can further delay the second ordinary sub-route of vehicle $m_2$ by an extra $\delta'$ to eliminate the overlap. This adjustment is equivalent to increasing the travel time of the preceding AV-enabled sub-route by $\delta'$.  
\input{figures/re-scheduling-vis-2}

To facilitate this more powerful re-scheduling strategy, we can simply relax the hard constraints of fixed travel time accumulation in \ref{frp1:c1} to \ref{frp1:c2} to the soft constraints shown in \ref{frp1:c8} to \ref{frp1:c11}.
\begin{align}
    &t^1_r \leq \overline{t}_m + \sum_{r': \; r' \, \prec \, r} \overline{\gamma}_{r'} \Delta t_{r'}, \qquad r \in R,  \label{frp1:c8} \\
    &t^1_r \geq \overline{t}_m + \sum_{r': \; r' \, \prec \, r} \underline{\gamma}_{r'} \Delta t_{r'}, \qquad r \in R, \label{frp1:c9} \\
    &t_r^2 \leq t_r^1 + \overline{\gamma}_r \Delta t_r, \qquad \forall \; r \in R, \label{frp1:c10} \\
    &t_r^2 \geq t_r^1 + \underline{\gamma}_r \Delta t_r, \qquad \forall \; r \in R.   \label{frp1:c11}
\end{align}

\section{The Re-Routing MILP} \label{Appendix C}
The re-routing MILP takes as input an expanded graph $G_e$, a set of customers $D$, an end time of operation $T$, and a set of \textit{infeasible time intervals} $Q$, where each interval has a start times $a_q$ and an end time $b_q$. During each infeasible time interval, the budget of remote controllers is already exhausted by other AVs. The objective is to find a minimum-cost tour visiting all customers exactly once. 

We adopt most notations from Section \ref{sec:milp} and overwrite the following for simplicity: Define a binary decision variable $x_{ij}$ for every edge $(i, j) \in E_e$ to indicate whether the edge is selected in the route. Define a binary decision variable $y_i$ for every customer $i \in D_e$ to indicate whether a duplicate of the customer in the expanded graph is served. Define a continuous variable $t_i \in [0, T]$ for every node $i \in V_e$ to represent the timestamp of the AV visiting the node. Define a binary decision variable $\alpha_{qij}$ for every $q \in Q$ and every edge $(i, j) \in E_e$, with $\alpha_{qij} = 1$ indicating the timestamp of the vehicle exiting edge $(i, j)$ is later than the start time of interval $q$. Define a binary decision variable $\beta_{qij}$ for every $q \in Q$ and every edge $(i, j) \in E_e$, with $\beta_{qij} = 1$ indicating the timestamp of the vehicle entering edge $(i, j)$ is earlier than the end time of interval $q$. Then, the re-routing MILP can be formulated as follows.  
\begin{mini}|s|[2]<b>
    {\bm{x}, \bm{y}, \bm{t}, \bm{\alpha}, \bm{\beta}} 
    {\sum_{(i, j) \in E^a_e} c_{ij}^1 x_{ij} + \sum_{(i, j) \in E^o_e} c_{ij}^2 x_{ij}}{}{}
    \addConstraint{\sum_{j \in \delta^-(i)} x_{ij}}{= \sum_{j \in \delta^+(i)} x_{ji}, \quad}{\forall \; i \in V_e - \{o, s\}}
    \addConstraint{\sum_{j \in \delta^-(o)} x_{oj} = \sum_{j \in \delta^+(s)} x_{js} \leq 1,}{}
    \addConstraint{\sum_{i \in D^V(j)} y_{i}}{= 1, \quad}{\forall \; j \in D}
    \addConstraint{\sum_{j \in \delta^+(i)} x_{ji}}{\geq y_{i}, \quad}{\forall \; i \in D_e}
    \addConstraint{x_{ij}}{\leq y_{i}, \quad}{\forall \; (i, j) \in A}
    \addConstraint{t_{j}}{\geq t_{i} + \Delta t_{ij} + T (x_{ij} - 1), \quad}{\forall \; (i, j) \in E_e}
    \addConstraint{t_{s}}{= t_{o} + \sum_{(i,j) \in E_e} \Delta t_{ij} x_{ij} \leq T}
    \addConstraint{t_{j}}{\leq a_q + T \alpha_{qij}, \quad}{\forall \; q \in Q, \forall \; (i, j) \in E^o_e}
    \addConstraint{t_{j}}{\geq a_q + T (\alpha_{qij} - 1), \quad}{\forall \; q \in Q, \forall \; (i, j) \in E^o_e}
    \addConstraint{b_q}{\leq t_{i} + T \beta_{qij}, \quad}{\forall \; q \in Q, \forall \; (i, j) \in E^o_e}
    \addConstraint{b_q}{\geq t_{i} + T (\beta_{qij} - 1), \quad}{\forall \; q \in Q, \forall \; (i, j) \in E^o_e}
    \addConstraint{\alpha_{qij} + \beta_{qij} + x_{ij}}{\leq 2, \quad}{\forall \; q \in Q, \forall \; (i, j) \in E^o_e}
\end{mini}

\section{Experimental Results} \label{Appendix D}
The rounded routing costs for the VRP-SA instances, which correspond to Figure \ref{fig:main_result}, are listed in Table \ref{tab:main_result} below.
\renewcommand{\arraystretch}{1.3}
\begin{longtable}{c|c|cc}
    \caption{Routing cost for the VRP-SA instances} \label{tab:main_result} \\ \hline
    \multirow{2}{*}{\textbf{Instances}} & \multirow{2}{*}{\textbf{Cost by HDVs}} & \multicolumn{2}{|c}{\textbf{Cost by HDVs + AVs}}  \\ \cline{3-4}
    & & \textbf{w.o. re-routing priority} & \textbf{w. re-routing priority} \\ \hline
    \endfirsthead

    \hline
    \multirow{2}{*}{\textbf{Instances}} & \multirow{2}{*}{\textbf{Cost by HDVs}} & \multicolumn{2}{|c}{\textbf{Cost by HDVs + AVs}}  \\ \cline{3-4}
    & & w.o. re-routing priority & w. re-routing priority \\ \hline
    \endhead

    \hline
    \endfoot

    \hline
    \endlastfoot

    P-n16-k8 & 600 & 387 & 386 \\
    P-n19-k2 & 280 & 251 & 247 \\
    P-n20-k2 & 282 & 249 & 249 \\
    P-n21-k2 & 282 & 258 & 258 \\
    P-n22-k2 & 288 & 269 & 243 \\
    P-n22-k8 & 750 & 508 & 472 \\
    P-n23-k8 & 694 & 474 & 475\\
    P-n40-k5 & 590 & 501 & 455\\
    P-n45-k5 & 650 & 563 & 532 \\
    P-n50-k7 & 700 & 645 & 579 \\
    P-n50-k8 & 798 & 627 & 624 \\
    P-n50-k10 & 880 & 682 & 680\\
    P-n51-k10 & 958 & 779 & 755 \\
    P-n55-k7 & 710 & 615 & 601\\
    P-n55-k10 & 870 & 685 & 689 \\
    P-n55-k15 & 1204 & 895 & 880\\
    P-n60-k10 & 926 & 876 & 856 \\
    P-n60-k15 & 1230 & 971 & 963 \\
    P-n65-k10 & 992 & 928 & 913 \\
    P-n70-k10 & 1034 & 988 & 926 \\
    P-n76-k4 & 742 & 759 & 722\\
    P-n76-k5 & 780 & 772 & 766 \\
    P-n101-k4 & 876 & 857 & 878 \\ 
\end{longtable}
\end{document}

%% file: figures/example_heuristic.tex
\begin{figure}[!htbp]
    \centering
    \begin{subfigure}[b]{0.45\linewidth}
    \centering
    \begin{tikzpicture}

    \Vertex[x = 0, y = 0, label = 0]{0}
    \Vertex[x = -1.5, y = 0, RGB, color = {190,174,212}, label = 1]{l1}
    \Vertex[x = -1.5, y = 1.5, color = white, label = 2]{l2}
    \Vertex[x = -1.5, y = 3, RGB, color = {190,174,212}, label = 3]{l3}
    \Vertex[x = -3, y = 3, RGB, color = {190,174,212}, label = 4]{l4}
    \Vertex[x = -3, y = 1.5, RGB, color = {190,174,212}, label = 5]{l5}
    \Vertex[x = -3, y = -1.5, color = white, label = 6]{l6}

    \Vertex[x = 1.5, y = 0, RGB, color = {190,174,212}, label = 7]{r1}
    \Vertex[x = 1.5, y = 1.5, color = white, label = 8]{r2}
    \Vertex[x = 1.5, y = 3, RGB, color = {190,174,212}, label = 9]{r3}
    \Vertex[x = 3, y = 3, RGB, color = {190,174,212}, label = 10]{r4}
    \Vertex[x = 3, y = 1.5, RGB, color = {190,174,212}, label = 11]{r5}
    \Vertex[x = 3, y = -1.5, color = white, label = 12]{r6}

    \Edge(0)(l1)
    \Edge(l1)(l2)
    \Edge(l2)(l3)
    \Edge(l3)(l4)
    \Edge(l4)(l5)
    \Edge(l5)(l2)
    \Edge(l5)(l6)

    \Edge(0)(r1)
    \Edge(r1)(r2)
    \Edge(r2)(r3)
    \Edge(r3)(r4)
    \Edge(r4)(r5)
    \Edge(r5)(r2)
    \Edge(r5)(r6)

    \Edge(l2)(r2)
    \Edge(l6)(r6)
    
    \end{tikzpicture}
    \caption{The road network}
    \label{fig:real_exp1}
    \end{subfigure}
    \hfill
    \begin{subfigure}[b]{0.45\linewidth}
    \centering
    \begin{tikzpicture}

    \Vertex[x = 0, y = 0, label = 0]{0}

    \Vertex[x = -1.5, y = 0, RGB, color = {190,174,212}, label = 1]{l1}
    \Vertex[x = -1.5, y = 1.5, color = white, opacity = 0.8, label = 2]{l2}
    \Vertex[x = -1.5, y = 3, RGB, color = {190,174,212}, label = 3]{l3}
    \Vertex[x = -3, y = 3, RGB, color = {190,174,212}, label = 4]{l4}
    \Vertex[x = -3, y = 1.5, RGB, color = {190,174,212}, label = 5]{l5}
    \Vertex[x = -3, y = -1.5, color = white, label = 6]{l6}

    \Vertex[x = 1.5, y = 0, RGB, color = {190,174,212}, label = 7]{r1}
    \Vertex[x = 1.5, y = 1.5, color = white, opacity = 0.8, label = 8]{r2}
    \Vertex[x = 1.5, y = 3, RGB, color = {190,174,212}, label = 9]{r3}
    \Vertex[x = 3, y = 3, RGB, color = {190,174,212}, label = 10]{r4}
    \Vertex[x = 3, y = 1.5, RGB, color = {190,174,212}, label = 11]{r5}
    \Vertex[x = 3, y = -1.5, color = white, label = 12]{r6}

    \Edge[Direct, color = red](0)(l1)
    \Edge[Direct, color = red](l1)(l3)
    \Edge[Direct, color = red](l3)(l4)
    \Edge[Direct, color = red](l4)(l5)
    \Edge(l5)(l6)

    \Edge[Direct, color = red](r1)(0)
    \Edge[Direct, color = red](r3)(r1)
    \Edge[Direct, color = red](r4)(r3)
    \Edge[Direct, color = red](r5)(r4)
    \Edge(r5)(r6)

    \Edge[Direct, color = red](l5)(r5)
    \Edge(l6)(r6)
    
    \end{tikzpicture}
    \caption{The optimal route}
    \label{fig:real_exp2}
    \end{subfigure}
    \caption{A realistic VRP instance in the real-world road network}
\end{figure}

%% file: figures/re-scheduling-visualization.tex
\begin{figure}[!htbp]
    \centering
    \begin{tikzpicture}

        \draw[thick] (-7, 3) -- (-7, 0.5);
        \draw[thick] (-7, -0.5) -- (-7, -3);
        \draw[thick] (8, 3) -- (8, 0.5);
        \draw[thick] (8, -0.5) -- (8, -3);

        \draw[dashed, thick] (-6, 1) -- (-6, -2);
        \draw[dashed, thick] (-2, 1) -- (-2, -2);
        \draw[dashed, thick] (2.5, 1) -- (2.5, -2);

        \draw[dashed, thick] (-6, -2) rectangle (-5, -2.5);
        \draw[dashed, thick] (-2, -2) rectangle (-1, -2.5);
        \draw[dashed, thick] (2.5, -2) rectangle (3.5, -2.5);

        \node[text width=1cm] at (-7,3.5) {$t = 0$};
        \node[text width=1cm] at (8,3.5) {$t = T$};
        \node[text width=1cm] at (-7.5, 2.25) {$m_1$};
        \node[text width=1cm] at (-7.5, 1.25) {$m_2$};
        \node[text width=1cm] at (-7.5, -1.25) {$m_1$};
        \node[text width=1cm] at (-7.5, -2.25) {$m_2$};

        \node[text width=2cm] at (0.2, 3) {Original};
        \node[text width=2cm] at (0, -0.5) {Rescheduled};

        \node[text width = 2cm] at (-5, -3) {$\underbracket{\hspace{1cm}}_{\displaystyle \delta}$};
        \node[text width = 2cm] at (-1, -3) {$\underbracket{\hspace{1cm}}_{\displaystyle \delta}$};
        \node[text width = 2cm] at (3.5, -3) {$\underbracket{\hspace{1cm}}_{\displaystyle \delta}$};

        \begin{scope}[on background layer]
        \filldraw [fill = vertexfill, draw = black] (-7, 2.5) rectangle (-5, 2);
        \filldraw [fill = vertexfill, draw = black] (-3, 2.5) rectangle (-1.5, 2);
        \filldraw [fill = vertexfill, draw = black] (0.5, 2.5) rectangle (3.5, 2);
        \filldraw [fill = vertexfill, draw = black] (6, 2.5) rectangle (7, 2);
        
        \filldraw [fill = vertexfill, draw = black] (-6, 1.5) rectangle (-4, 1);
        \filldraw [fill = vertexfill, draw = black] (-2, 1.5) rectangle (-1, 1);
        \filldraw [fill = vertexfill, draw = black] (2.5, 1.5) rectangle (4, 1);
        
        \filldraw [fill = vertexfill, draw = black] (-7, -1) rectangle (-5, -1.5);
        \filldraw [fill = vertexfill, draw = black] (-3, -1) rectangle (-1.5, -1.5);
        \filldraw [fill = vertexfill, draw = black] (0.5, -1) rectangle (3.5, -1.5);
        \filldraw [fill = vertexfill, draw = black] (6, -1) rectangle (7, -1.5);
        
        \filldraw [fill = Orchid, draw = black] (-5, -2) rectangle (-3, -2.5);
        \filldraw [fill = Orchid, draw = black] (-1, -2) rectangle (0, -2.5);
        \filldraw [fill = Orchid, draw = black] (3.5, -2) rectangle (5, -2.5);

        \end{scope}

        
    \end{tikzpicture}
    \caption{The illustration of the re-scheduling FRP with fixed travel time}
    \label{fig:frp1}
\end{figure}

%% file: figures/example_stsp.tex
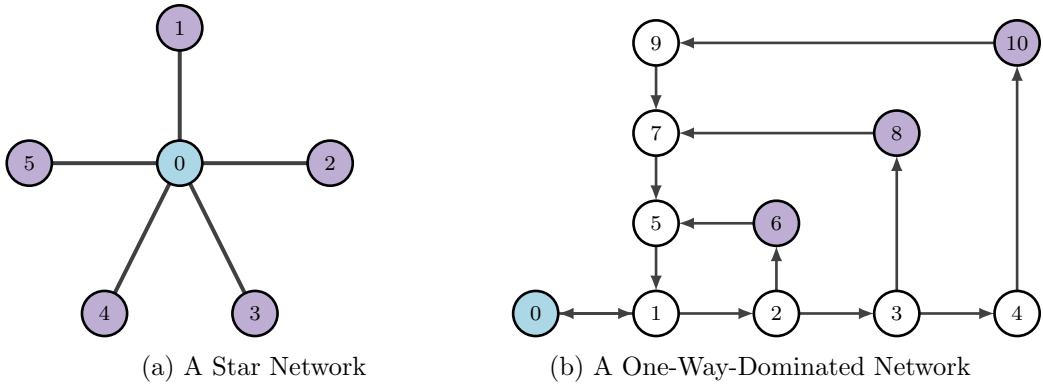
\begin{figure}[!htbp]
    \centering
    \begin{subfigure}[b]{0.4\textwidth}
        \begin{tikzpicture}
        \Vertex[x = 0, y = 1.8, label = 1, RGB, color = {190,174,212}]{1}
        \Vertex[x = 0, y = 0, label = 0]{0}
        \Vertex[x = 2, y = 0, label = 2, RGB, color = {190,174,212}]{2}
        \Vertex[x = 1, y = -2, label = 3, RGB, color = {190,174,212}]{3}
        \Vertex[x = -1, y = -2, label = 4, RGB, color = {190,174,212}]{4}
        \Vertex[x = -2, y = 0, label = 5, RGB, color = {190,174,212}]{5}
    
        \Edge(0)(1)
        \Edge(0)(2)
        \Edge(0)(3)
        \Edge(0)(4)
        \Edge(0)(5)
        
        \end{tikzpicture}
        \caption{A Star Network}
        \label{fig:star}
    \end{subfigure}
    \begin{subfigure}[b]{0.4\textwidth}
        \begin{tikzpicture}
        \Vertex[x = -4, y = -1.8, label = 0]{0}
        \Vertex[x = -2.4, y = -1.8, label = 1, color = white]{1}
        \Vertex[x = -0.8, y = -1.8, label = 2, color = white]{2}
        \Vertex[x = 0.8, y = -1.8, label = 3, color = white]{3}
        \Vertex[x = 2.4, y = -1.8, label = 4, color = white]{4}
        \Vertex[x = -2.4, y = -0.6, label = 5, color = white]{5}
        \Vertex[x = -0.8, y = -0.6, label = 6, RGB, color = {190,174,212}]{6}
        \Vertex[x = -2.4, y = 0.6, label = 7, color = white]{7}
        \Vertex[x = 0.8, y = 0.6, label = 8, RGB, color = {190,174,212}]{8}
        \Vertex[x = -2.4, y = 1.8, label = 9, color = white]{9}
        \Vertex[x = 2.4, y = 1.8, label = 10, RGB, color = {190,174,212}]{10}
        
        \Edge[lw = 1pt, Direct](0)(1)
        \Edge[lw = 1pt, Direct](1)(0)
        \Edge[lw = 1pt, Direct](1)(2)
        \Edge[lw = 1pt, Direct](2)(3)
        \Edge[lw = 1pt, Direct](3)(4)
        
        \Edge[lw = 1pt, Direct](2)(6)
        \Edge[lw = 1pt, Direct](6)(5)
        \Edge[lw = 1pt, Direct](5)(1)

        \Edge[lw = 1pt, Direct](3)(8)
        \Edge[lw = 1pt, Direct](8)(7)
        \Edge[lw = 1pt, Direct](7)(5)
        
        \Edge[lw = 1pt, Direct](4)(10)
        \Edge[lw = 1pt, Direct](10)(9)
        \Edge[lw = 1pt, Direct](9)(7)
        
        \end{tikzpicture}
        \caption{A One-Way-Dominated Network}
        \label{fig:one-way}
    \end{subfigure}
    \caption{Two types of networks to prove tight upper bounds in Lemma \ref{lemma:1} and Lemma \ref{lemma:2}}
    \label{fig:example_network}
\end{figure}

%% file: figures/re-scheduling-vis-2.tex
\begin{figure}[!htbp]
    \centering
    \begin{tikzpicture}

        \draw[thick] (-7, 3) -- (-7, 0.5);
        \draw[thick] (-7, -0.5) -- (-7, -3);
        \draw[thick] (8, 3) -- (8, 0.5);
        \draw[thick] (8, -0.5) -- (8, -3);

        \draw[dashed, very thick] (-6, 1) -- (-6, -2);
        \draw[dashed, very thick] (-2, 1) -- (-2, -2);
        \draw[dashed, very thick] (2.5, 1) -- (2.5, -2);

        \draw[pattern=north east lines, pattern color = vertexfill] (-1.5, 2.5) rectangle (-0.5, 2);
        \draw[pattern=north east lines, pattern color = vertexfill] (-1.5, -1) rectangle (-0.5, -1.5);

        \draw[dashed, thick] (-6, -2) rectangle (-5, -2.5);
        \draw[dashed, thick] (-2, -2) rectangle (-0.5, -2.5);
        \draw[dashed, thick] (2.5, -2) rectangle (3.5, -2.5);
        \draw[dashed, thick, Orchid, pattern=north east lines, pattern color = Orchid] (-1, -2) rectangle (-0.5, -2.5);

        \node[text width=1cm] at (-7,3.5) {$t = 0$};
        \node[text width=1cm] at (8,3.5) {$t = T$};
        \node[text width=1cm] at (-7.5, 2.25) {$m_1$};
        \node[text width=1cm] at (-7.5, 1.25) {$m_2$};
        \node[text width=1cm] at (-7.5, -1.25) {$m_1$};
        \node[text width=1cm] at (-7.5, -2.25) {$m_2$};

        \node[text width=2cm] at (0.2, 3) {Original};
        \node[text width=2cm] at (0, -0.5) {Rescheduled};
        
        \node[text width = 2cm] at (-5, -3) {$\underbracket{\hspace{1cm}}_{\displaystyle \delta}$};
        \node[text width = 2cm] at (-1, -3) {$\underbracket{\hspace{1.5cm}}_{\displaystyle \delta + \delta'}$};
        \node[text width = 2cm] at (3.5, -3) {$\underbracket{\hspace{1cm}}_{\displaystyle \delta}$};
        
        \begin{scope}[on background layer]
        \filldraw [fill = vertexfill, draw = black] (-7, 2.5) rectangle (-5, 2);
        \filldraw [fill = vertexfill, draw = black] (-3, 2.5) rectangle (-1.5, 2);
        \filldraw [fill = vertexfill, draw = black] (0.5, 2.5) rectangle (3.5, 2);
        \filldraw [fill = vertexfill, draw = black] (6, 2.5) rectangle (7, 2);
        
        \filldraw [fill = vertexfill, draw = black] (-6, 1.5) rectangle (-4, 1);
        \filldraw [fill = vertexfill, draw = black] (-2, 1.5) rectangle (-1, 1);
        \filldraw [fill = vertexfill, draw = black] (2.5, 1.5) rectangle (4, 1);
        
        \filldraw [fill = vertexfill, draw = black] (-7, -1) rectangle (-5, -1.5);
        \filldraw [fill = vertexfill, draw = black] (-3, -1) rectangle (-1.5, -1.5);
        \filldraw [fill = vertexfill, draw = black] (0.5, -1) rectangle (3.5, -1.5);
        \filldraw [fill = vertexfill, draw = black] (6, -1) rectangle (7, -1.5);
        
        \filldraw [fill = Orchid, draw = black] (-5, -2) rectangle (-3, -2.5);
        \filldraw [fill = Orchid, draw = black] (-0.5, -2) rectangle (0.5, -2.5);
        \filldraw [fill = Orchid, draw = black] (3.5, -2) rectangle (5, -2.5);

        \end{scope}

        
    \end{tikzpicture}
     \caption{The illustration of the re-scheduling FRP with flexible travel time}
    \label{fig:frp1-1}
\end{figure}
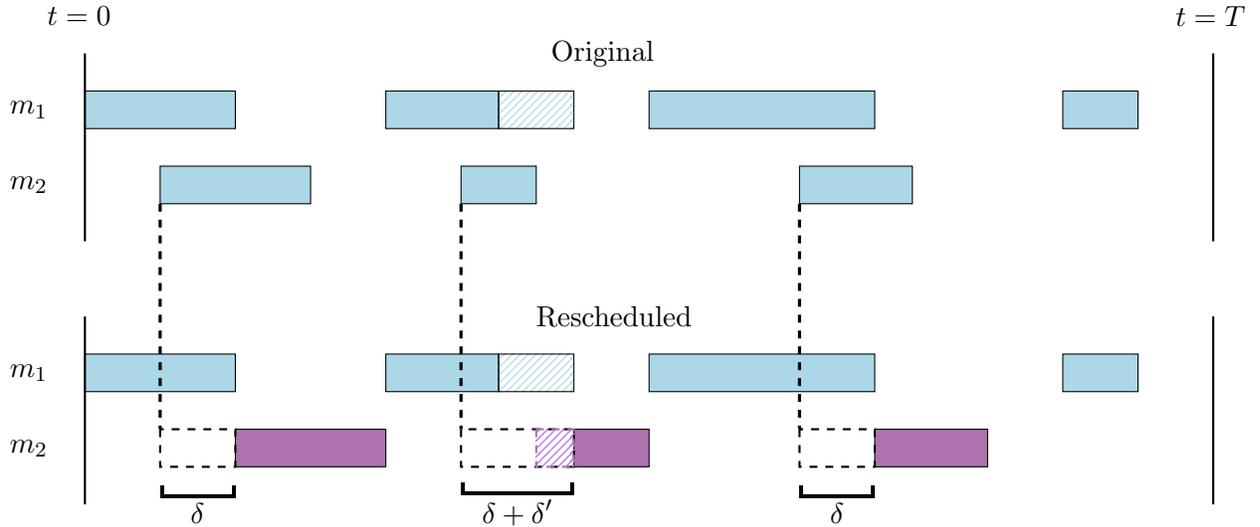

%% file: references.bib
@article{VIDAL2022105643,
title = {Hybrid genetic search for the CVRP: Open-source implementation and SWAP* neighborhood},
journal = {Computers \& Operations Research},
volume = {140},
pages = {105643},
year = {2022},
issn = {0305-0548},
author = {Thibaut Vidal}
}

@article{miller1960integer,
  title={Integer programming formulation of traveling salesman problems},
  author={Miller, Clair E and Tucker, Albert W and Zemlin, Richard A},
  journal={Journal of the ACM (JACM)},
  volume={7},
  number={4},
  pages={326--329},
  year={1960},
  publisher={ACM New York, NY, USA}
}

@article{RODRIGUEZPEREIRA2019615,
title = {The Steiner Traveling Salesman Problem and its extensions},
journal = {European Journal of Operational Research},
volume = {278},
number = {2},
pages = {615-628},
year = {2019},
author = {Jessica Rodríguez-Pereira and Elena Fernández and Gilbert Laporte and Enrique Benavent and Antonio Martínez-Sykora},
}

@article{MOLINA2020103745,
title = {The heterogeneous vehicle routing problem with time windows and a limited number of resources},
journal = {Engineering Applications of Artificial Intelligence},
volume = {94},
pages = {103745},
year = {2020},
author = {Jose C. Molina and Jose L. Salmeron and Ignacio Eguia and Jesus Racero},
}

@article{PRINS2009916,
title = {Two memetic algorithms for heterogeneous fleet vehicle routing problems},
journal = {Engineering Applications of Artificial Intelligence},
volume = {22},
number = {6},
pages = {916-928},
year = {2009},
note = {Artificial Intelligence Techniques for Supply Chain Management},
author = {Christian Prins},
}

@article{IMRAN2009509,
title = {A variable neighborhood-based heuristic for the heterogeneous fleet vehicle routing problem},
journal = {European Journal of Operational Research},
volume = {197},
number = {2},
pages = {509-518},
year = {2009},
author = {Arif Imran and Said Salhi and Niaz A. Wassan},
}

@article{LI20101111,
title = {An adaptive memory programming metaheuristic for the heterogeneous fixed fleet vehicle routing problem},
journal = {Transportation Research Part E: Logistics and Transportation Review},
volume = {46},
number = {6},
pages = {1111-1127},
year = {2010},
author = {Xiangyong Li and Peng Tian and Y.P. Aneja},
}

@article{baldacci2009unified,
  title={A unified exact method for solving different classes of vehicle routing problems},
  author={Baldacci, Roberto and Mingozzi, Aristide},
  journal={Mathematical Programming},
  volume={120},
  pages={347--380},
  year={2009},
  publisher={Springer}
}

@book{toth2014vehicle,
  title={Vehicle routing: problems, methods, and applications},
  author={Toth, Paolo and Vigo, Daniele},
  year={2014},
  publisher={SIAM}
}

@article{braekers2016vehicle,
  title={The vehicle routing problem: State of the art classification and review},
  author={Braekers, Kris and Ramaekers, Katrien and Van Nieuwenhuyse, Inneke},
  journal={Computers \& Industrial Engineering},
  volume={99},
  pages={300--313},
  year={2016},
  publisher={Elsevier}
}

@article{lin2014survey,
  title={Survey of green vehicle routing problem: past and future trends},
  author={Lin, Canhong and Choy, King Lun and Ho, George TS and Chung, Sai Ho and Lam, HY},
  journal={Expert Systems with Applications},
  volume={41},
  number={4},
  pages={1118--1138},
  year={2014},
  publisher={Elsevier}
}

@article{kocc2016thirty,
  title={Thirty years of heterogeneous vehicle routing},
  author={Ko{\c{c}}, {\c{C}}a{\u{g}}r{\i} and Bekta{\c{s}}, Tolga and Jabali, Ola and Laporte, Gilbert},
  journal={European Journal of Operational Research},
  volume={249},
  number={1},
  pages={1--21},
  year={2016},
  publisher={Elsevier}
}

@article{montoya2015literature,
  title={A literature review on the vehicle routing problem with multiple depots},
  author={Montoya-Torres, Jairo R and Franco, Juli{\'a}n L{\'o}pez and Isaza, Santiago Nieto and Jim{\'e}nez, Heriberto Felizzola and Herazo-Padilla, Nilson},
  journal={Computers \& Industrial Engineering},
  volume={79},
  pages={115--129},
  year={2015},
  publisher={Elsevier}
}

@article{braysy2005vehicle_2,
  title={Vehicle routing problem with time windows, Part II: Metaheuristics},
  author={Br{\"a}ysy, Olli and Gendreau, Michel},
  journal={Transportation Science},
  volume={39},
  number={1},
  pages={119--139},
  year={2005},
  publisher={INFORMS}
}

@article{braysy2005vehicle_1,
  title={Vehicle routing problem with time windows, Part I: Route construction and local search algorithms},
  author={Br{\"a}ysy, Olli and Gendreau, Michel},
  journal={Transportation science},
  volume={39},
  number={1},
  pages={104--118},
  year={2005},
  publisher={INFORMS}
}

@article{VIDAL2014658,
title = {A unified solution framework for multi-attribute vehicle routing problems},
journal = {European Journal of Operational Research},
volume = {234},
number = {3},
pages = {658-673},
year = {2014},
author = {Thibaut Vidal and Teodor Gabriel Crainic and Michel Gendreau and Christian Prins},
}

@article{uchoa2017new,
  title={New benchmark instances for the capacitated vehicle routing problem},
  author={Uchoa, Eduardo and Pecin, Diego and Pessoa, Artur and Poggi, Marcus and Vidal, Thibaut and Subramanian, Anand},
  journal={European Journal of Operational Research},
  volume={257},
  number={3},
  pages={845--858},
  year={2017},
  publisher={Elsevier}
}

@article{augerat1995computational,
  title={Computational results with a branch and cut code for the capacitated vehicle routing problem},
  author={Augerat, Philippe and Naddef, D and Belenguer, JM and Benavent, E and Corberan, A and Rinaldi, Giovanni},
  year={1995}
}

@article{solomon1987algorithms,
  title={Algorithms for the vehicle routing and scheduling problems with time window constraints},
  author={Solomon, Marius M},
  journal={Operations Research},
  volume={35},
  number={2},
  pages={254--265},
  year={1987},
  publisher={Informs}
}

@inproceedings{li2021learning,
title={Learning to delegate for large-scale vehicle routing},
author={Sirui Li and Zhongxia Yan and Cathy Wu},
booktitle={Advances in Neural Information Processing Systems},
editor={A. Beygelzimer and Y. Dauphin and P. Liang and J. Wortman Vaughan},
year={2021}
}

@article{nazari2018reinforcement,
  title={Reinforcement learning for solving the vehicle routing problem},
  author={Nazari, Mohammadreza and Oroojlooy, Afshin and Snyder, Lawrence and Tak{\'a}c, Martin},
  journal={Advances in Neural Information Processing Systems},
  volume={31},
  year={2018}
}

@inproceedings{
kool2018attention,
title={Attention, Learn to Solve Routing Problems!},
author={Wouter Kool and Herke van Hoof and Max Welling},
booktitle={International Conference on Learning Representations},
year={2019}
}

@article{fukasawa2006robust,
  title={Robust branch-and-cut-and-price for the capacitated vehicle routing problem},
  author={Fukasawa, Ricardo and Longo, Humberto and Lysgaard, Jens and Arag{\~a}o, Marcus Poggi de and Reis, Marcelo and Uchoa, Eduardo and Werneck, Renato F},
  journal={Mathematical programming},
  volume={106},
  pages={491--511},
  year={2006},
  publisher={Springer}
}

@article{fischetti1994branch,
  title={A branch-and-bound algorithm for the capacitated vehicle routing problem on directed graphs},
  author={Fischetti, Matteo and Toth, Paolo and Vigo, Daniele},
  journal={Operations Research},
  volume={42},
  number={5},
  pages={846--859},
  year={1994},
  publisher={INFORMS}
}

@article{desrochers1992new,
  title={A new optimization algorithm for the vehicle routing problem with time windows},
  author={Desrochers, Martin and Desrosiers, Jacques and Solomon, Marius},
  journal={Operations Research},
  volume={40},
  number={2},
  pages={342--354},
  year={1992},
  publisher={INFORMS}
}

@article{pisinger2007general,
  title={A general heuristic for vehicle routing problems},
  author={Pisinger, David and Ropke, Stefan},
  journal={Computers \& Operations Research},
  volume={34},
  number={8},
  pages={2403--2435},
  year={2007},
  publisher={Elsevier}
}

@article{bell2004ant,
  title={Ant colony optimization techniques for the vehicle routing problem},
  author={Bell, John E and McMullen, Patrick R},
  journal={Advanced Engineering Informatics},
  volume={18},
  number={1},
  pages={41--48},
  year={2004},
  publisher={Elsevier}
}

@article{homberger1999two,
  title={Two evolutionary metaheuristics for the vehicle routing problem with time windows},
  author={Homberger, J{\"o}rg and Gehring, Hermann},
  journal={INFOR: Information Systems and Operational Research},
  volume={37},
  number={3},
  pages={297--318},
  year={1999},
  publisher={Taylor \& Francis}
}

@article{bramel1995location,
  title={A location based heuristic for general routing problems},
  author={Bramel, Julien and Simchi-Levi, David},
  journal={Operations Research},
  volume={43},
  number={4},
  pages={649--660},
  year={1995},
  publisher={INFORMS}
}

@article{fisher1981generalized,
  title={A generalized assignment heuristic for vehicle routing},
  author={Fisher, Marshall L and Jaikumar, Ramchandran},
  journal={Networks},
  volume={11},
  number={2},
  pages={109--124},
  year={1981},
  publisher={Wiley Online Library}
}

@article{gillett1974heuristic,
  title={A heuristic algorithm for the vehicle-dispatch problem},
  author={Gillett, Billy E and Miller, Leland R},
  journal={Operations Research},
  volume={22},
  number={2},
  pages={340--349},
  year={1974},
  publisher={INFORMS}
}

@article{schneider2014electric,
  title={The electric vehicle-routing problem with time windows and recharging stations},
  author={Schneider, Michael and Stenger, Andreas and Goeke, Dominik},
  journal={Transportation Science},
  volume={48},
  number={4},
  pages={500--520},
  year={2014},
  publisher={INFORMS}
}

@article{erdougan2012green,
  title={A green vehicle routing problem},
  author={Erdo{\u{g}}an, Sevgi and Miller-Hooks, Elise},
  journal={Transportation Research Part E: Logistics and Transportation Review},
  volume={48},
  number={1},
  pages={100--114},
  year={2012},
  publisher={Elsevier}
}

@article{renaud1996tabu,
  title={A tabu search heuristic for the multi-depot vehicle routing problem},
  author={Renaud, Jacques and Laporte, Gilbert and Boctor, Fayez F},
  journal={Computers \& Operations Research},
  volume={23},
  number={3},
  pages={229--235},
  year={1996},
  publisher={Elsevier}
}

@article{golden1984fleet,
  title={The fleet size and mix vehicle routing problem},
  author={Golden, Bruce and Assad, Arjang and Levy, Larry and Gheysens, Filip},
  journal={Computers \& Operations Research},
  volume={11},
  number={1},
  pages={49--66},
  year={1984},
  publisher={Elsevier}
}

@article{savelsbergh1995general,
  title={The general pickup and delivery problem},
  author={Savelsbergh, Martin WP and Sol, Marc},
  journal={Transportation Science},
  volume={29},
  number={1},
  pages={17--29},
  year={1995},
  publisher={INFORMS}
}

@article{kolen1987vehicle,
  title={Vehicle routing with time windows},
  author={Kolen, Antoon WJ and Rinnooy Kan, AHG and Trienekens, Harry WJM},
  journal={Operations Research},
  volume={35},
  number={2},
  pages={266--273},
  year={1987},
  publisher={INFORMS}
}

@article{lenstra1981complexity,
  title={Complexity of vehicle routing and scheduling problems},
  author={Lenstra, Jan Karel and Kan, AHG Rinnooy},
  journal={Networks},
  volume={11},
  number={2},
  pages={221--227},
  year={1981},
  publisher={Wiley Online Library}
}

@article{clarke1964scheduling,
  title={Scheduling of vehicles from a central depot to a number of delivery points},
  author={Clarke, Geoff and Wright, John W},
  journal={Operations Research},
  volume={12},
  number={4},
  pages={568--581},
  year={1964},
  publisher={Informs}
}

@article{dantzig1959truck,
  title={The truck dispatching problem},
  author={Dantzig, George B and Ramser, John H},
  journal={Management Science},
  volume={6},
  number={1},
  pages={80--91},
  year={1959},
  publisher={Informs}
}

@Online{enride2022,
 author = {Ed Garsten},
 year = {2022},
 title = {Einride Gets Go-Ahead For Driverless Electric Trucks On Public Roads},
 journal = {Forbes},
 url = {https://www.forbes.com/sites/edgarsten/2022/06/23/einride-gets-go-ahead-for-driverless-trucks-on-public-roads/?sh=1b6be34ac0c7}
}

@Online{Waymophoenix,
 author = {ANDREW J. HAWKINS},
 year = {2022},
 title = {Waymo’s driverless vehicles are picking up passengers in downtown Phoenix},
 journal = {The Verge},
 url = {https://www.theverge.com/2022/8/29/23323593/waymo-driverless-vehicles-passengers-downtown-phoenix}
}

@article{HABOUCHA201737,
title = {User preferences regarding autonomous vehicles},
journal = {Transportation Research Part C: Emerging Technologies},
volume = {78},
pages = {37-49},
year = {2017},
issn = {0968-090X},
author = {Chana J. Haboucha and Robert Ishaq and Yoram Shiftan},
}

@Online{Honda2021,
 author = {Colin Beresford},
 year = {2021},
 title = {Honda Legend Sedan with Level 3 Autonomy Available for Lease in Japan},
 journal = {Car and Drive},
 url = {https://www.caranddriver.com/news/a35729591/honda-legend-level-3-autonomy-leases-japan/}
}

@article{sae2018taxonomy,
  title={Taxonomy and definitions for terms related to driving automation systems for on-road motor vehicles},
  author={Sae International},
  journal={SAE international},
  volume={4970},
  number={724},
  pages={1--5},
  year={2018}
}

@techreport{leslie2022analysis,
author={Leslie, Alex and Murray, Dan},
title={An Analysis of the Operational Costs of Trucking},
institution = {American Transportation Research Institute},
year={2022}
}
